\documentclass[times]{nlaauth}

\usepackage{moreverb}

\usepackage[colorlinks,bookmarksopen,bookmarksnumbered,citecolor=red,urlcolor=red]{hyperref}

\usepackage[]{algorithm2e}

\newcommand\BibTeX{{\rmfamily B\kern-.05em \textsc{i\kern-.025em b}\kern-.08em
T\kern-.1667em\lower.7ex\hbox{E}\kern-.125emX}}

\def\volumeyear{2014}

\numberwithin{equation}{section}

\begin{document}

\runningheads{Q. Deng and V. Ginting}{A Post-processing Technique}

\title{A Post-processing Technique for Streamline Upwind/Petrov-Galerkin for Advection Dominated Partial Differential Equations}

\author{Quanling Deng  \corrauth \footnotemark[2] and Victor Ginting \footnotemark[2] }

\address{Department of Mathematics, University of Wyoming, 1000 E University Ave, Laramie, WY 82071, U.S.A.}

\corraddr{Quanling Deng, Department of Mathematics, University of Wyoming, 1000 E University Ave, Laramie, WY 82071, U.S.A.}

\begin{abstract}
We consider the construction of locally conservative fluxes by means of a simple post-processing technique obtained from the finite element solutions of advection diffusion equations. It is known that a naive calculation of fluxes from these solutions yields non-conservative fluxes. We consider two finite element methods: the usual continuous Galerkin finite element (CGFEM) for solving non dominating advection diffusion equations and the streamline upwind/Petrov-Galerkin (SUPG) for solving advection dominated problems. We then describe the post-processing technique for constructing conservative fluxes from the numerical solutions of the general variational formulation.  The post-processing technique requires solving an auxiliary Neumann boundary value problem on each element independently  and it produces a locally conservative flux on a vertex centered dual mesh relative to the finite element mesh.  We provide a convergence analysis for the post-processing technique. Performance of the technique and the convergence behavior are demonstrated through numerical examples including a set of test problems for advection diffusion equations, advection dominated equations, and drift-diffusion equations.
\end{abstract}


\maketitle

\footnotetext[2]{Email: qdeng@uwyo.edu, vginting@uwyo.edu}

\vspace{-6pt}

\section{Introduction} \label{sec:intro}

\noindent
The continuous Galerkin finite element method (CGFEM) is perhaps the most the popular method for solving partial differential equations and it has many advantages for the numerical simulations \cite{brenner2008mathematical}. A significant disadvantage of this method, however, is the lack of local conservation of quantities, such as pressure, mass, momentum, and energy. Numerical schemes for a wide range of applications are required to satisfy a local conservation property. For example, in a multiphase flow problem in porous media, local conservation of flux associated with the pressure is required to simulate the transport quantity. Negative concentrations and other non-physical results can be generated if the standard CGFEM is applied to solve the governing equations because the direct construction of Darcy velocity from the standard CGFEM solution produces normal fluxes which are not continuous across the boundary of each element, i.e., the velocity approximation is not locally conservative over mesh elements and divergence of the velocity field  over an element does not balance the discrete mass accumulation \cite{povich2013finite}. 
Another example is the drift-diffusion equations, which model the motion of electrons and holes in semiconductor materials \cite{mock1983analysis, markowich1986stationary}. It is desirable that the numerical schemes for the drift-diffusion yield solutions that satisfy local conservation of electron and hole current densities, and in addition, the numerical schemes also need to maintain certain stability criteria when the drift velocities dominate their diffusivities \cite{bochev2013new}. 

Due to its local conservation property, the finite volume method (FVM) is widely used for solving the governing equations coming from a conservation law in fluid dynamics \cite{leveque2002finite} and electrodynamics \cite{taflove2000computational}. For advection dominated equations or problems where internal or boundary layers may be produced, FVMs are developed \cite{manzini2008finite, bochev2013new}. A naive calculation of the FVM solutions gives locally conservative fluxes since the FVM solutions are governed by a set of local equations representing the local conservation. The theories and properties for lower order FVMs are well established \cite{eymard2000finite,  leveque2002finite} but the design and analysis on high order FVMs are still under investigation \cite{cai2003development, chen2010new, chen2012higher}. Both low order and high order CGFEMs are well established and understood \cite{brenner2008mathematical}. They are much desired due to their flexibility on handling complicated geometries and the
corresponding positive definite linear algebraic system associated with the methods, which in turn can be solved by typical iterative techniques in a straight forward manner. Thus, if one is to keep these advantages and at the same time requiring
numerical quantities that are locally conservative, then it is necessary to devise a method to obtain local conservative fluxes from the standard finite element methods.

The work in \cite{cockburn2007locally} develops a hybrid CGFEM method 
for pure diffusion problems
which includes a two-step post-processing to obtain the locally conservative flux. The first step is the direct computation of a numerical flux trace defined on edges or faces while the second step is a local element-by-element post-processing procedure based on the Raviart-Thomas projection. An enrichment-type method within the framework of CGFEM for both pure diffusion and advection diffusion problems is presented in \cite{hughes2000continuous} to obtain the the locally conservative flux by adding elementwise constant functions to the finite element space.  The work in \cite{sun2009locally} is also based on enriching the finite element space.  
A characteristics-mixed finite element method which is locally conservative for advection dominated transport problem is presented in \cite{arbogast1995characteristics}. The method uses a characteristic approximation for the advection in time combined with a low-order mixed finite element spatial approximation. 

A locally conservative Galerkin method is proposed for advection diffusion equations in \cite{nithiarasu2004simple}. A post-processing, which is a simple averaging procedure, is used to develop the method to recover the conservative fluxes on individual elements independently. A continuing work of \cite{nithiarasu2004simple} is presented in \cite{thomas2008element, thomas2008locally}. For instance, in \cite{thomas2008element}, the proposed locally conservative Galerkin method requires solving a system of simultaneous equations with weakly imposed Neumann boundary conditions on each element and hence the global matrix inversion is avoided. It is proved to be equivalent to the standard global CGFEM and yields locally conservative fluxes while maintaining the advantages of CGFEM. This method is applied to solve the two-phase flow problems in porous media in \cite{zhang2013locally}.

A global post-processing technique is introduced in \cite{toledo1989mixed} to obtain the conservative stresses in the context of stress recovery in two dimensional elasticity problems. A more general form of the technique is investigated in \cite{loula1995higher}, in which global and element-by-element post-processing techniques for finite element method are proposed within the framework of displacement methods and based on least-square residuals of the equilibrium equation and irrotationality condition. 
Another global post-processing technique is presented in \cite{srivastava1992three}. 

A local post-processing technique based on dual meshes and stream functions is presented in \cite{cordes1992continuous}. In two dimensional setting, each element is divided into four subtriangles to create the dual mesh and local conservative flux is obtained on dual mesh by introducing new and improved velocities on the subtriangles where the new subtriangle velocities are introduced to satisfy the local conservation property.

Another post-processing technique for CGFEM is investigated in \cite{larson2004conservative} recently. The technique obtains an elementwise conservative approximations of fluxes by modifying the CGFEM solutions which is based on the idea of enriching the solution space by adding additional discontinuous degrees of freedom using piecewise constant jump corrections. This technique is applied to gain locally conservative velocity approximations for stabilized finite element methods for solving advection dominated diffusion problems arising from variably saturated groundwater flow \cite{kees2008locally} and saltwater intrusion \cite{povich2013finite}.

In this paper, a post-processing technique is proposed for recovering the locally conservative fluxes from two finite element methods: CGFEM for solving non-dominating advection diffusion problems and SUPG for solving advection dominated problems. SUPG method is first promoted by Hughes {\it et al} in 1982 \cite{brooks1982streamline}; it introduces a certain amount of artificial diffusion in the streamline direction to give a proper stabilization when strong advection is present.  A common drawback of SUPG is that the amount of the artificial diffusion should be carefully selected by user and the amount is controlled by the stabilization parameter $\delta.$ Discussion on how to choose the parameter is found in \cite{brooks1982streamline, tezduyar2000finite, tezduyar2003stabilization}.


A recently proposed post-processing technique for CGFEM for pure diffusion elliptic problems is presented in \cite{bush2013application}, in which auxiliary elemental Neumann problems with the boundary conditions coming from the CGFEM solutions is proposed to recover locally conservative fluxes. These local problems yield low dimensional linear algebra systems which are independent of each other.  The local conservation property is satisfied on the control volumes, which are built from a vertex centered dual mesh. In particular for two dimensional triangular element setting, it is constructed by connecting the barycenter of the triangle and its middle points of the three edges. The post-processing technique proposed in \cite{bush2013application} is simple and the required extra computation is relatively inexpensive. 
In this paper, we aim at applying this technique to advection diffusion problems as well as advection dominated problems. Due to the existence of the advection term, the technique can not be applied directly. Firstly, the post-processing technique strongly depends on the SUPG/CGFEM formulation and the properties of the linear basis to ensure the compatibility condition for the local Neumann problem. One of the properties is that the sum of the basis functions on an element equals one and the other is that the gradient of the sum equals zero. To ensure the compatibility condition, the technique requires the SUPG/CGFEM formulation to have gradient operator on the test functions. Consequently, the technique requires integration by parts on both the diffusion term and the advection term. For the advection dominated equation, besides this requirement, the technique also requires the stabilization term to contain the gradient operator on the test functions. Secondly, for the construction of the local Neumann problem, the lower order term which is the advection term in the advection diffusion equation is considered as a data represented by the SUPG/CGFEM solution, which is then subtracted from the forcing term.

The rest of the paper is organized as follows. Section \ref{sec:supg} contains a concise description of a general formulation of SUPG/CGFEM where CGFEM is for the steady advection diffusion equation and SUPG is for the steady advection dominated diffusion equation as well as an obvious but important fact of the formulation.
Section \ref{sec:pp} presents the detailed post-processing technique for the general formulation. 
Section \ref{sec:ana} shows an analysis of the technique to establish the convergence of the post-processed quantity to the true one. Finally, numerical examples are given to demonstrate the performance of the technique, to confirm the convergence rates as well as to illustrate the validity of the local conservation in Section \ref{sec:num}.

\section{Streamline Upwind/Petrov-Galerkin Method} \label{sec:supg}
We consider the advection diffusion equation
\begin{equation}\label{pde}
\begin{cases}
\begin{aligned}
 \nabla \cdot ( -k\nabla u + \boldsymbol{v} u ) & =f \quad \text{in} \quad \Omega, \\
u&= g \quad \text{on} \quad \partial\Omega, \\
\end{aligned}
\end{cases}
\end{equation}

\noindent
where $\Omega$ is a bounded open domain in $\mathbb{R}^2$ with Lipschitz boundary $\partial\Omega$,  $k$ is the elliptic coefficient, $\boldsymbol{v}$ is the advective velocity, $u$ is the function to be found, $f$ is a forcing function. Assuming $k_{\max} \geq k \geq k_{\min} >0$ for all $\boldsymbol{x} \in\Omega$ and $\nabla \cdot \boldsymbol{v} \geq 0$,  Lax-Milgram theorem guarantees a unique weak solution to \eqref{pde}; see \cite{hundsdorfer2003numerical}.  For the polygonal domain $\Omega$, we consider a partition $\mathcal{T}_h$ consisting of triangular elements $\tau$ such that $\overline\Omega = \bigcup_{\tau\in\mathcal{T}_h} \tau.$ We set $h=\max_{\tau\in\mathcal{T}_h} h_\tau$ where $h_\tau$ is defined as the diameter of $\tau$. The linear finite element space is defined as
$$
V_h = \{w_h\in C(\overline\Omega): w_h|_\tau \ \text{is linear for all} \ \tau\in\mathcal{T}_h \ \text{and} \ w_h|_{\partial\Omega} = 0 \} \subset H^1_0(\Omega).
$$
The SUPG formulation for \eqref{pde} is to find $u_h$ with $(u_h - g_h) \in V_h$, such that
\begin{equation} \label{eq:supg}
a(u_h, w_h) = \ell(w_h) \quad \forall \ w_h \in V_h, 
\end{equation}
where 
\begin{equation*}
a(u_h, w_h) = \int_\Omega ( k \nabla u_h - u_h  \boldsymbol{v} ) \cdot \nabla w_h \ \text{d} \boldsymbol{x} + \int_\Omega \delta  \big(  \nabla \cdot ( -k\nabla u_h + \boldsymbol{v} u_h ) \big)  ( \boldsymbol{v}  \cdot \nabla w_h)  \ \text{d} \boldsymbol{x},
\end{equation*}
and
\begin{equation*}
\ell(w_h) = \int_\Omega f \big(  w_h + \delta \boldsymbol{v}  \cdot \nabla w_h \big)  \ \text{d} \boldsymbol{x}.
\end{equation*}
Here $g_h$ can be thought of as the interpolant of $g$ using the finite element basis.

Notice that when $\delta = 0$, \eqref{eq:supg} is reduced to the usual continuous Galerkin finite element method (CGFEM). When $k$ is much smaller compared to the advective velocity $\boldsymbol{v}$, \eqref{pde} is an advection dominated diffusion equation, in which case the CGFEM solutions develop spurious oscillations \cite{johnson2009numerical}.  SUPG is a popular and efficient remedy which adds artificial dissipation in the direction of velocity $\boldsymbol{v}$.  The coefficient $\delta$ is the SUPG stabilization parameter which controls the amount of artificial dissipation; discussion on how to choose the parameter can be found in \cite{brooks1982streamline, tezduyar2000finite, tezduyar2003stabilization}.

We will use \eqref{eq:supg} as a general formulation for both SUPG and CGFEM to describe the post-processing technique in Section \ref{sec:pp}. In preparation for describing the post-processing technique in Section \ref{sec:pp}, we present an important but obvious fact of the formulation. Let $Z$ be all vertices in the partition $\mathcal{T}_h$ with $Z = Z_{\text{in}} \cup Z_\text{d}$, where $Z_{\text{in}}$ is the set of interior vertices and $Z_\text{d}$ is the set of vertices on $\partial\Omega$. Denoting the basis of $V_h$ as $\{ \phi_z \}_{z \in Z_\text{in}}$, according to
\eqref{eq:supg} $u_h$ satisfies
\begin{equation} \label{eq:supgphi}
a(u_h, \phi_z) = \ell (\phi_z) \quad \forall \ z \in Z_{\text{in}}.
\end{equation}

\noindent
For a $z \in Z_{\text{in}}$, let $\Omega^z$ be the support of the basis function $\phi_z$. Then in the partition $\mathcal{T}_h$, $\Omega^z = \cup_{i=1}^N \tau_i^z$, where
$\tau_i^z$ is an element that has $z$ as one of its vertices, and
 $N$ is the total number of such elements. We then write
\begin{equation} \label{eq:Q}
\begin{aligned}
a(u_h, \phi_z) & = \int_{\Omega^z} \Big(
( k\nabla u_h - \boldsymbol{v} u_h )  \cdot \nabla \phi_z + \delta  \big(  \nabla \cdot ( -k\nabla u_h + \boldsymbol{v} u_h ) \big)  ( \boldsymbol{v}  \cdot \nabla \phi_z)
\Big) \ \text{d} \boldsymbol{x} \\
& = \sum_{i=1}^N \int_{\tau_i^z}   \Big( ( k\nabla u_h - \boldsymbol{v} u_h )  \cdot \nabla \phi_z + \delta  \big(  \nabla \cdot ( -k\nabla u_h + \boldsymbol{v} u_h ) \big)  ( \boldsymbol{v}  \cdot \nabla \phi_z)  \Big) \ \text{d} \boldsymbol{x} \\
& := \sum_{i=1}^N Q_i^z, \\
\ell(\phi_z) & = \int_{\Omega^z} f\big( \phi_z + \delta \boldsymbol{v}  \cdot \nabla \phi_z \big) \ \text{d} \boldsymbol{x} = \sum_{i=1}^N  \int_{\tau_i^z} f \big( \phi_z + \delta \boldsymbol{v}  \cdot \nabla \phi_z \big) \ \text{d} \boldsymbol{x}  := \sum_{i=1}^N F_i^z.
\end{aligned}
\end{equation}
Thus \eqref{eq:supgphi} becomes
\begin{equation} \label{eq:QF}
\sum_{i=1}^N Q_i^z = \sum_{i=1}^N F_i^z.
\end{equation}
Equation \eqref{eq:QF} is important and we will use this fact to derive the corresponding post-processing technique in Section \ref{sec:pp}.

\section{A Post-processing technique}  \label{sec:pp}

\noindent
In this section, we propose a post-processing technique for SUPG formulation \eqref{eq:supg} which naturally includes the CGFEM.

\subsection{Auxiliary Elemental Neumann Problem} \label{sec:bvp}

We want to construct a locally conservative flux over a control volume from SUPG
(and also  CGFEM) solutions.
We create a dual mesh relative to the finite element mesh to generate a set of control volumes, each of which is associated with vertex $z$ (see the middle plot of Figure \ref{fig:taucv}) and on which we would like the flux conservation to hold. We notice that without post-processing, $\boldsymbol{\nu}_h = -k \nabla u_h + \boldsymbol{v}u_h$ is discontinuous at the boundary of each element which violates the conservation for the flux. With post-processing, we hope to get $\widetilde {\boldsymbol{\nu}}_h \cdot \boldsymbol{n}$ that is continuous at the boundary of each control volumes and satisfies a local conservation in the sense
\begin{equation} \label{eq:cvconservation}
\int_{\partial C^z} \widetilde {\boldsymbol{\nu}}_h \cdot \boldsymbol{n}  \ \text{d} l = \int_{C^z} f \ \text{d} \boldsymbol{x}. 
\end{equation}

\begin{figure}[ht]
\centering
\includegraphics[height=4.5cm]{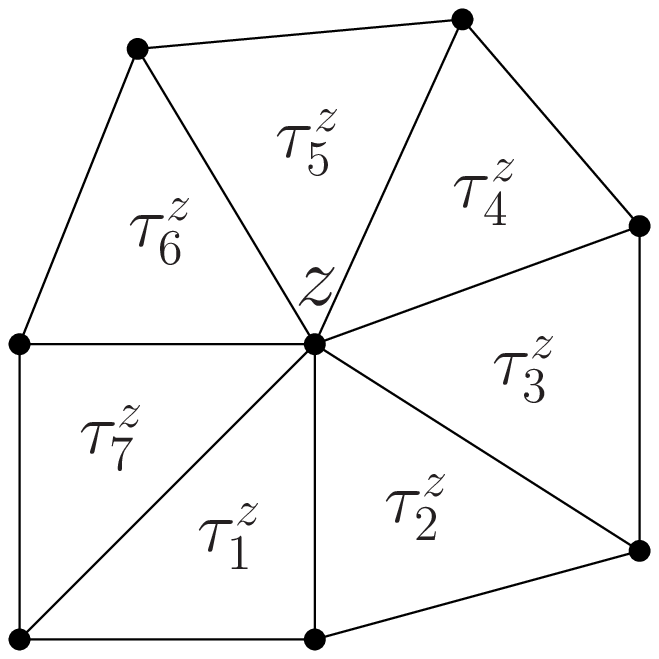}  \hspace*{0.45cm}
\includegraphics[height=4.5cm]{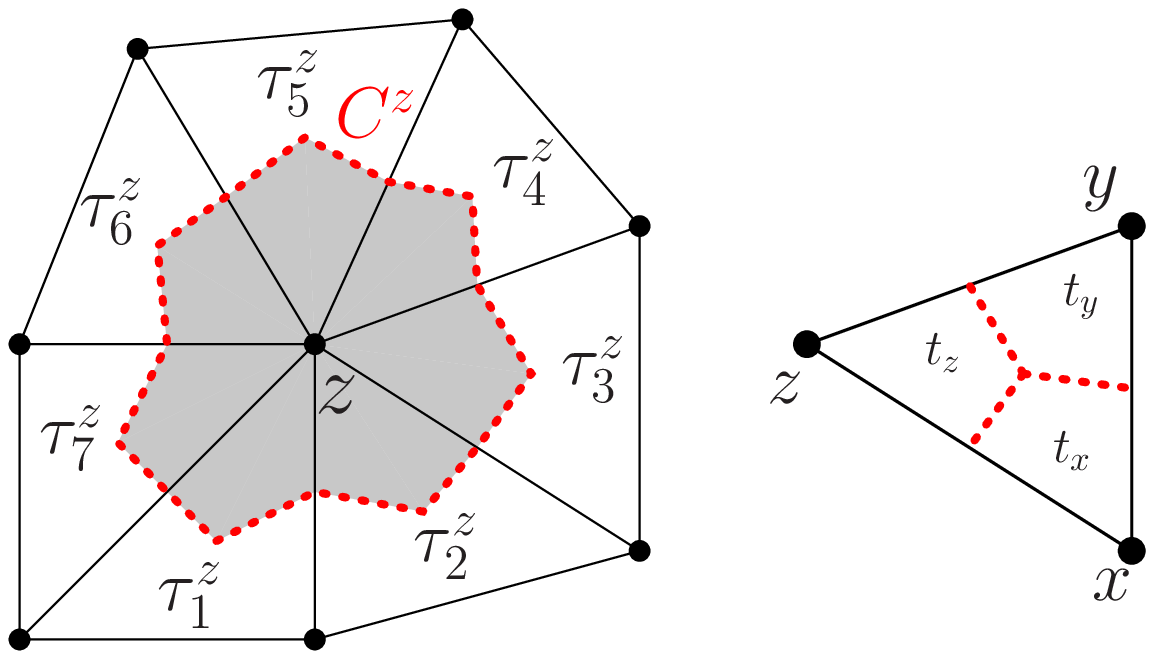}
\caption{$\Omega^z$: support of $\phi_z$ (left), control volume $C^z$ (middle), a finite element $\tau$ (right) }\label{fig:taucv}
\end{figure}

We set and solve a Neumann problem on $\tau$ using finite volume type method to obtain the conservative fluxes on each control volume. As presented in the right plot of Figure \ref{fig:taucv}, we define the set of vertices $v(\tau) = \{ x, y, z \}$ in the triangular element. We discretize each element $\tau$ into three non-overlapping quadrilaterals $t_x, t_y,$ and $t_z$. Let $ \xi \in v(\tau) $, in each quadrilateral $t_\xi$, and we decompose $\partial t_\xi $ as $\partial t_\xi = ( \partial \tau \cap \partial t_\xi ) \cup ( \partial C^\xi \cap \partial t_\xi ).$  The auxiliary problem for the post-processing is to find $\widetilde u_{\tau, h} \in V(\tau) =\text{span}\{ \phi_z, z \in v(\tau)\}$ satisfying 

\begin{equation} \label{eq:bvp}
- \int_{\partial t_\xi } k \nabla \widetilde u_{\tau, h} \cdot \boldsymbol{n}  \ \text{d} l = \int_{t_\xi} f \ \text{d} \boldsymbol{x} - \int_{\partial t_\xi } u_h \boldsymbol{v} \cdot \boldsymbol{n}  \ \text{d} l  \quad \forall \ \xi\in v(\tau),
\end{equation}
where $u_h$ is the solution of \eqref{eq:supg}. The boundary condition satisfies
\begin{equation} \label{eq:bvpb}
 \int_{\partial \tau \cap \partial t_\xi } ( - k \nabla \widetilde u_{\tau, h}  +  u_h \boldsymbol{v} ) \cdot \boldsymbol{n}  \ \text{d} l =
\int_{\partial \tau \cap \partial t_\xi } \widetilde g_\tau  \ \text{d} l
= F^\xi - Q^\xi \quad \forall \ \xi\in v(\tau),
\end{equation}

\noindent
where $F^\xi = \ell_\tau(\phi_\xi)$ defined as $ \ell(\cdot)$ restricted to $\tau$ and $Q^\xi = a_\tau(u_h, \phi_\xi)$ defined as $a(\cdot, \cdot)$ restricted to $\tau.$ The existence and uniqueness of the solution to this Neumann problem are established by the following lemma.

\newtheorem{lem}{Lemma}[section]
\begin{lem}  \label{lem:maxerr}
The fully Neumann problem \eqref{eq:bvp}-\eqref{eq:bvpb} has a unique solution up to a constant and the post-processed flux equals the true flux over the element $\tau$, i.e.,
\begin{equation} \label{eq:fluxed}
\int_{\partial\tau } ( -k \nabla \widetilde u_{\tau, h} + u_h \boldsymbol{v} ) \cdot \boldsymbol{n}  \ \text{d} l = \int_{\partial\tau } ( -k \nabla u + u \boldsymbol{v} ) \cdot \boldsymbol{n}  \ \text{d} l.
\end{equation}
\end{lem}

\begin{proof}
The existence and the uniqueness of the solution are established by verifying the compatibility condition \cite{evans2010partial}. To verify the compatibility condition, we calculate 
\begin{equation*}
\int_{\partial\tau} \widetilde g_\tau \ \text{d} l = \int_{\partial \tau} ( - k \nabla \widetilde u_{\tau, h}  +  u_h \boldsymbol{v} ) \cdot \boldsymbol{n}  \ \text{d} l 
= \sum_{\xi\in v(\tau)} F^\xi - \sum_{\xi\in v(\tau)} Q^\xi.
\end{equation*}

\noindent
By using \eqref{eq:supg}, $\sum_{\xi\in v(\tau)} \phi_\xi = 1$, and $\nabla \big( \sum_{\xi\in v(\tau)} \phi_\xi \big) = 0$, we obtain
\begin{equation*}
\begin{aligned}
\sum_{\xi\in v(\tau)} F^\xi & = \sum_{\xi\in v(\tau)} \int_\tau f \big(  \phi_\xi + \delta \boldsymbol{v}  \cdot \nabla \phi_\xi  \big) \ \text{d} \boldsymbol{x} \\
& = \int_\tau f \Big( \sum_{\xi\in v(\tau)} \phi_\xi + \delta \boldsymbol{v}  \cdot \nabla ( \sum_{\xi\in v(\tau)} \phi_\xi )  \Big) \ \text{d} \boldsymbol{x} \\
& = \int_\tau f \ \text{d} \boldsymbol{x},
\end{aligned} 
\end{equation*}
and
\begin{equation*}
\begin{aligned}
\sum_{\xi\in v(\tau)} Q^\xi & = \sum_{\xi\in v(\tau)} \int_\tau \Big( ( k \nabla u_h - u_h \boldsymbol{v} ) \cdot \nabla \phi_\xi + \delta  \big(  \nabla \cdot ( -k\nabla u_h + \boldsymbol{v} u_h ) \big)  ( \boldsymbol{v}  \cdot \nabla \phi_\xi) \Big) \ \text{d} \boldsymbol{x} \\
& = \int_\tau \Big( ( k \nabla u_h - u_h \boldsymbol{v} ) \cdot \nabla \big( \sum_{\xi\in v(\tau)} \phi_\xi \big) + \delta  \big(  \nabla \cdot ( -k\nabla u_h + \boldsymbol{v} u_h ) \big)  \big( \boldsymbol{v}  \cdot \nabla \big( \sum_{\xi\in v(\tau)} \phi_\xi \big) \big) \ 
\Big)
\text{d} \boldsymbol{x} \\
& = 0.
\end{aligned} 
\end{equation*}
Thus the compatibility condition $\int_{\partial\tau} \widetilde g_\tau \ \text{d} l = \int_\tau f \ \text{d} \boldsymbol{x}$ is confirmed. As a byproduct, using \eqref{pde} and divergence theorem, we have the identity
\begin{equation*}
\begin{aligned}
\int_{\partial \tau} ( - k \nabla \widetilde u_{\tau, h}  +  u_h \boldsymbol{v} ) \cdot \boldsymbol{n} \ \text{d} l & = \int_{\partial\tau} \widetilde g_\tau \ \text{d} l = \int_\tau f \ \text{d} \boldsymbol{x}  = \int_\tau \nabla \cdot ( -k\nabla u + \boldsymbol{v} u )  \ \text{d} \boldsymbol{x} \\
& = \int_{\partial\tau } ( -k \nabla u + u \boldsymbol{v} ) \cdot \boldsymbol{n}  \ \text{d} l.
\end{aligned} 
\end{equation*}
This completes the proof.
\end{proof}

\subsection{Local Linear Algebra System } \label{sec:llas}
Next we describe the resulting linear system that is coming from
\eqref{eq:bvp}:
\begin{equation} \label{eq:locsys}
\begin{aligned}
- \int_{ \partial C^x \cap \partial t_x } k \nabla \widetilde u_{\tau, h} \cdot \boldsymbol{n}  \ \text{d} l & = Q^x - F^x + \int_{t_x} f \ \text{d} \boldsymbol{x} - \int_{ \partial C^x \cap \partial t_x  } u_h \boldsymbol{v} \cdot \boldsymbol{n}  \ \text{d} l, \\
- \int_{ \partial C^y \cap \partial t_y } k  \nabla \widetilde u_{\tau, h} \cdot \boldsymbol{n}  \ \text{d} l & = Q^y - F^y + \int_{t_y} f \ \text{d} \boldsymbol{x} - \int_{\partial C^y \cap \partial t_y } u_h \boldsymbol{v} \cdot \boldsymbol{n}  \ \text{d} l, \\
- \int_{ \partial C^z \cap \partial t_z } k \nabla \widetilde u_{\tau, h} \cdot \boldsymbol{n}  \ \text{d} l & = Q^z - F^z + \int_{t_z} f \ \text{d} \boldsymbol{x} - \int_{\partial C^z \cap \partial t_z} u_h \boldsymbol{v} \cdot \boldsymbol{n}  \ \text{d} l. \\
\end{aligned}
\end{equation}
Since $\widetilde u_{\tau, h}$ is represented as
\begin{equation} \label{eq:ppsol}
\widetilde u_{\tau, h} = \alpha_x \phi_x + \alpha_y \phi_y + \alpha_z \phi_z,
\end{equation}
we use it in \eqref{eq:locsys} to get a three dimensional linear system
$$
A \boldsymbol{\alpha} = \boldsymbol{b},
$$
where $\boldsymbol{\alpha} = (\alpha_x, \alpha_y, \alpha_z)^\top $,
$ \boldsymbol{b} = (b_x, b_y, b_z)^\top $ with 
$$
b_\xi = Q^\xi - F^\xi + \int_{t_\xi} f \ \text{d} \boldsymbol{x} - \int_{ \partial C^\xi \cap \partial t_\xi  } u_h \boldsymbol{v} \cdot \boldsymbol{n}  \ \text{d} l \quad \forall \ \xi\in v(\tau),
$$
and
$$
A_{\xi\eta} = -\int_{ \partial C^\xi \cap \partial t_\xi  } k \nabla \phi_\eta \cdot \boldsymbol{n} \ \text{d} l \quad \forall\ \xi, \eta \in v(\tau).
$$
As discussed Section \ref{sec:bvp}, the solution of \eqref{eq:bvp} is unique up to a constant, thus this system is singular and there are infinitely many solutions. However, since the desired final outcome is flux, which is a derivative information, it will be unique.

\subsection{Local Conservation}
\noindent
In Sections \ref{sec:bvp} and \ref{sec:llas}, the local auxiliary Neumann problem and its corresponding linear system are proposed with the hope of obtaining the required conservation on control volumes. The following lemma verifies the conservation property \eqref{eq:cvconservation} on control volumes whose edge does not overlap the global boundary.

\begin{lem}  \label{lem:localconserv}
The desired local conservation property \eqref{eq:cvconservation} on the control volume $C^z$ where $z\in Z_{\text{in}}$ is satisfied.
\end{lem}

\begin{proof}
Let $\Omega^z = \cup_{i=1}^N \tau_i^z$ be the subdomain associated with vertex $z$. Using for instance the last equation in \eqref{eq:locsys}, we move the last term
on the right hand side into the left-hand side to yield
\begin{equation} \label{eq:locelemequ}
 \int_{ \partial C^z \cap \partial t_z } \big( - k\nabla \widetilde u_{\tau, h} \cdot \boldsymbol{n}  + u_h \boldsymbol{v} \cdot \boldsymbol{n}  \big) \ \text{d} l = Q^z - F^z + \int_{t_z} f \ \text{d} \boldsymbol{x}.
\end{equation}
Consider the $N$ elements that have $z$ as a vertex and these $N$ elements lead to $N$ equations  like \eqref{eq:locelemequ}. By summing up these $N$ equations left-hand side by left-hand side and right-hand side by right-hand side and using \eqref{eq:QF}, we then have
\begin{equation*}
 \int_{ \partial C^z } \big( - k \nabla \widetilde u_{\tau, h} \cdot \boldsymbol{n}  + u_h \boldsymbol{v} \cdot \boldsymbol{n} \big) \ \text{d} l = \sum_{j=1}^N Q^z_j - \sum_{j=1}^N F^z_j + \int_{C^z} f \ \text{d} \boldsymbol{x}  = \int_{C^z} f \ \text{d} \boldsymbol{x}.
\end{equation*}
This confirms \eqref{eq:cvconservation}, i.e, the local conservation is satisfied on the control volume $C^z.$
\end{proof}

\section{An Error Analysis for the post-processing}  \label{sec:ana}
In this section, we focus on establishing a convergence property of the post-processing solution $\widetilde u_{\tau, h}$ in $H^1$ semi-norm by following the analysis framework in \cite{bush2013application}. We start with some general properties in the linear finite element space $V(\tau)$ and then establish a lemma by using those properties. In \cite{bush2013application}, the $H^1$ semi-norm error of the post-processed solution is bounded by mainly dealing with the diffusion term while here to establish the bound of the error, we need to deal with both the diffusion and advection terms. Lemma \ref{lem:elemerr} established the boundedness of the error, followed by the main theorem for the post-processed solution. We establish the convergence property on the post-processing for SUPG, which naturally implies the convergence property of the post-processing for CGFEM.

Let $V(\tau) = \text{span}\{ \phi_z, z \in v(\tau)\}$ be the linear finite element space over the element $\tau$ and $V^0(\tau)$ be the space of piecewise constant functions on element $\tau$ such that $V^0(\tau) = \text{span} \{ \psi_\xi \}_{\xi \in v(\tau)}$ with 
$$
\psi_\xi(\boldsymbol{x}) = \Big\lbrace\begin{array}{c} 1 \quad \text{if} \quad \boldsymbol{x} \in t_\xi  \\ 0 \quad \text{if} \quad \boldsymbol{x} \notin t_\xi \end{array}.
$$

\noindent
Define a map $I_\tau: V(\tau) \rightarrow V^0(\tau)$.  This map has a property \cite{chou2000error, chatzipantelidis2002finite} that 

\begin{equation} \label{eq:interr}
\| w - I_\tau w \|_{L^2(\tau)} \leq C h_\tau \| \nabla w \|_{L^2(\tau)}, \quad \text{for} \quad w \in V(\tau).
\end{equation}
Given $w\in V(\tau)$, we multiply \eqref{eq:bvp} by $w_\xi$ (the nodal value of $w$ at $\xi$) and sum over $\xi \in v(\tau)$ to yield 

\begin{equation} \label{eq:bvpsum}
-\sum_{\xi\in v(\tau)} \int_{\partial t_\xi} k \nabla \widetilde u_{\tau, h}  \cdot \boldsymbol{n} I_\tau w \ \text{d} l = \int_\tau f I_\tau w \ \text{d} \boldsymbol{x} -  \sum_{\xi\in v(\tau)} \int_{\partial t_\xi } u_h \boldsymbol{v} \cdot \boldsymbol{n}    I_\tau w \ \text{d} l
\end{equation}
and from \eqref{eq:bvpb} we get

\begin{equation} \label{eq:bvpbsum}
\int_{\partial \tau} ( - k \nabla \widetilde u_{\tau, h} \cdot \boldsymbol{n} + u_h \boldsymbol{v} \cdot \boldsymbol{n} ) I_\tau w \ \text{d} l  = \int_{\partial \tau} \widetilde g_\tau  I_\tau w \ \text{d} l,
\end{equation}
where
\begin{equation} \label{eq:poly}
\widetilde g_\tau = \sum_{\xi \in v(\tau)} ( F^\xi - Q^\xi ) P_\xi,
\end{equation}
and $P_\xi$ are polynomials of degree two in $\tau$ satisfying
\begin{equation*}
\int_{\partial\tau} P_\xi \psi_\eta \ \text{d} l = \delta_{\xi \eta}  \qquad \text{and}  \qquad \int_{\partial\tau} P_\xi \phi_\eta \ \text{d} l = \delta_{\xi \eta},
\end{equation*}
with $ \delta_{\xi \eta}$ the usual Kronecker delta. 
An example of such a polynomial is given in \cite{bush2013application}.
With these properties, it is easy to verify that $\widetilde{g}_\tau$ satisfies the boundary condition \eqref{eq:bvpb}.

\begin{lem} \label{lem:link}
Assume $w\in V(\tau)$ and $\widetilde g_\tau $ is defined in \eqref{eq:poly}. Then
\begin{equation} \label{eq:pass1}
\begin{aligned}
\int_{\partial \tau} ( k \nabla u \cdot \boldsymbol{n} - u \boldsymbol{v} \cdot \boldsymbol{n} + \widetilde g_\tau) w \ \text{d} l  = a_\tau(u - u_h, w) \\
\end{aligned}
\end{equation}
and
\begin{equation} \label{eq:pass2}
\int_{\partial \tau} \widetilde g_\tau( w - I_\tau w)  \ \text{d} l = 0,
\end{equation}
\end{lem}
\noindent
where $u$ is the true solution and $u_h$ is the solution of \eqref{eq:supg}.

\begin{proof}
Let $w\in V(\tau)$ so that $w = \displaystyle \sum_{\eta \in v(\tau)} w_\eta \phi_\eta$.
With $\widetilde{g}_\tau$ as defined in \eqref{eq:poly}, we calculate
\begin{equation*}
\begin{aligned}
\int_{\partial \tau} \widetilde g_\tau w \ \text{d} l & = \int_{\partial \tau} \widetilde g_\tau \Big(\sum_{\eta \in v(\tau)} w_\eta \phi_\eta \Big) \ \text{d} l
= \sum_{\eta \in v(\tau)} w_\eta  \int_{\partial \tau} \widetilde g_\tau \phi_\eta \ \text{d} l  \\
& = \sum_{\xi \in v(\tau)} ( F^\xi - Q^\xi ) \sum_{\eta \in v(\tau)} w_\eta  \int_{\partial \tau} P_\xi  \phi_\eta \ \text{d} l  \\
& = \sum_{\eta \in v(\tau)} w_\eta F^\eta - \sum_{\eta \in v(\tau)} w_\eta Q^\eta \\
& = \sum_{\eta \in v(\tau)} w_\eta \int_\tau f \big(\phi_\eta + \delta \boldsymbol{v}  \cdot \nabla \phi_\eta \big) \ \text{d} \boldsymbol{x}  \\
& \quad - \sum_{\eta \in v(\tau)} w_\eta \int_\tau \Big( (k \nabla u_h - u_h \boldsymbol{v}) \cdot \nabla \phi_\eta + \delta  \big(  \nabla \cdot ( -k\nabla u_h + \boldsymbol{v} u_h ) \big)  ( \boldsymbol{v}  \cdot \nabla \phi_\eta)  \Big) \ \text{d} \boldsymbol{x}  \\
& = \int_\tau f \big(w + \delta \boldsymbol{v}  \cdot \nabla w \big) \ \text{d} \boldsymbol{x} \\
& \quad - \int_\tau \Big( ( k \nabla u_h - u_h \boldsymbol{v} ) \cdot \nabla w + \delta  \big(  \nabla \cdot ( -k\nabla u_h + \boldsymbol{v} u_h ) \big)  ( \boldsymbol{v}  \cdot \nabla w) \Big) \ \text{d} \boldsymbol{x} \\
& = \ell_\tau (w) - a_\tau(u_h, w).
\end{aligned}
\end{equation*}

\noindent
Now multiplying the PDE in \eqref{pde} by a test function $w\in V(\tau)$ and using integration by parts, we obtain
\begin{equation} \label{eq:pdeint}
\int_\tau (k \nabla u - u \boldsymbol{v} ) \cdot \nabla w \ \text{d} \boldsymbol{x} = \int_\tau f w \ \text{d} \boldsymbol{x} + \int_{\partial\tau} (k \nabla u - u \boldsymbol{v} ) \cdot \boldsymbol{n} w \ \text{d} l.
\end{equation}

\noindent
By using \eqref{eq:pdeint} and \eqref{pde}, we then have
\begin{equation*}
\begin{aligned}
\int_{\partial \tau} ( k \nabla u \cdot \boldsymbol{n} - u \boldsymbol{v} \cdot \boldsymbol{n} + \widetilde g_\tau) w \ \text{d} l  & = \int_{\partial \tau} (k \nabla u - u \boldsymbol{v} ) \cdot \boldsymbol{n} w \ \text{d} l  + \ell_\tau (w) - a_\tau(u_h, w) \\
& = a_\tau (u-u_h, w),
\end{aligned}
\end{equation*}

\noindent
which completes \eqref{eq:pass1}. Moreover since $I_\tau w = \displaystyle \sum_{\eta \in v(\tau)} w_\eta \psi_\eta$, we have
\begin{equation*}
\begin{aligned}
\int_{\partial \tau} \widetilde g_\tau I_\tau w \ \text{d} l & = \int_{\partial \tau} \widetilde g_\tau \sum_{\eta \in v(\tau)} w_\eta \psi_\eta  \ \text{d} l
= \sum_{\eta \in v(\tau)} w_\eta  \int_{\partial \tau} \widetilde g_\tau \psi_\eta \ \text{d} l  \\
& = \sum_{\xi \in v(\tau)} ( F^\xi - Q^\xi )  \sum_{\eta \in v(\tau)} w_\eta  \int_{\partial \tau} P_\xi  \psi_\eta \ \text{d} l  \\
& = \sum_{\eta \in v(\tau)} w_\eta F^\eta - \sum_{\eta \in v(\tau)} w_\eta Q^\eta \\
& = \sum_{\eta \in v(\tau)} w_\eta  \int_\tau f \big(\phi_\eta + \delta \boldsymbol{v}  \cdot \nabla \phi_\eta \big) \ \text{d} \boldsymbol{x}  \\
& \quad - \sum_{\eta \in v(\tau)} w_\eta \int_\tau  \Big( (k \nabla u_h - u_h \boldsymbol{v}) \cdot \nabla \phi_\eta + \delta  \big(  \nabla \cdot ( -k\nabla u_h + \boldsymbol{v} u_h ) \big)  ( \boldsymbol{v}  \cdot \nabla \phi_\eta) \Big) \ \text{d} \boldsymbol{x}  \\
& = \int_\tau f \big(w + \delta \boldsymbol{v}  \cdot \nabla w \big) \ \text{d} \boldsymbol{x} \\
& \quad - \int_\tau \Big( ( k \nabla u_h - u_h \boldsymbol{v} ) \cdot \nabla w + \delta  \big(  \nabla \cdot ( -k\nabla u_h + \boldsymbol{v} u_h ) \big)  ( \boldsymbol{v}  \cdot \nabla w) \Big) \ \text{d} \boldsymbol{x} \\
& = \ell_\tau (w) - a_\tau(u_h, w).
\end{aligned}
\end{equation*}

\noindent
Therefore
$$
\int_{\partial \tau} \widetilde g_\tau( w - I_\tau w)  \ \text{d} l = 0,
$$
which completes the proof.
\end{proof}

Now we are ready to present and prove the lemma and theorem for the error
of the post-processes solution in $H^1$ semi-norm. The tools for the proof mainly consist of triangle inequality, integration by parts, divergence theorem, and subtracting-adding a term.

\begin{lem}  \label{lem:elemerr}
Let $u$ be the solution of \eqref{pde} and $\widetilde u_{\tau,h}$ be the post-proccessed solution \eqref{eq:ppsol} on the element $\tau$. For sufficiently smooth $k$ in $\tau$ and sufficiently small $h_\tau$, the error $(u - \widetilde u_{\tau, h}) $ satisfies 
$$
\| \nabla( u - \widetilde u_{\tau, h} ) \|_{L^2(\tau)} \leq C_1 R_{1,\tau} + C_2 R_{2,\tau} + C_3R_{3,\tau},
$$

\noindent
where
\begin{equation*}
\begin{aligned}
R_{1,\tau} & = \| \nabla( u - w ) \|_{L^2(\tau)}, \\
R_{2,\tau} & = h_\tau \big(  \| f \|_{L^2(\tau)} +  \| \nabla w \|_{L^2(\tau)}  + \| \boldsymbol{v} \|_{L^2(\tau)} + \| \nabla \cdot (u_h\boldsymbol{v}) \|_{L^2(\tau)}  \big), \\
R_{3,\tau} & = \| \nabla(u - u_h) \|_{L^2(\tau)},\\
\end{aligned}
\end{equation*}
and $w \in V(\tau)$, and $C_1, C_2, C_3$ are constants independent of $h_\tau$.

\end{lem}

\begin{proof}
By letting $e_{\tau, h} = w - \widetilde u_{\tau, h}$
and $w \in V(\tau)$, triangle inequality gives
\begin{equation} \label{eq:p0}
\| \nabla( u - \widetilde u_{\tau, h} ) \|_{L^2(\tau)} \leq \| \nabla( u - w ) \|_{L^2(\tau)} + \| \nabla( w - \widetilde u_{\tau, h} ) \|_{L^2(\tau)} =
R_{1,\tau} + \| \nabla e_{\tau,h} \|_{L^2(\tau)}. 
\end{equation}

\noindent
Using \eqref{eq:pdeint}, $e_{\tau, h} = (w - u) + (u - \widetilde u_{\tau, h})$ and $e_{\tau, h} = e_{\tau, h} - I_\tau e_{\tau, h} + I_\tau e_{\tau, h}$, we then get
\begin{equation} \label{eq:p1}
k_{\min, \tau} \| \nabla e_{\tau, h} \|^2_{L^2(\tau)} \leq \int_\tau k \nabla e_{\tau, h} \cdot \nabla e_{\tau, h} \ \text{d} \boldsymbol{x} = J_1 + J_2 + J_3,
\end{equation}

\noindent
where
\begin{equation*}
\begin{aligned}
J_1 & = \int_\tau k \nabla (w-u) \cdot \nabla e_{\tau, h} \ \text{d} \boldsymbol{x}, \\
J_2 & = \int_\tau f ( e_{\tau, h} - I_\tau e_{\tau, h} ) \ \text{d} \boldsymbol{x}, \\
J_3 & = \int_\tau f I_\tau e_{\tau, h} \ \text{d} \boldsymbol{x} + \int_{\partial\tau} ( k\nabla u - u\boldsymbol{v})  \cdot \boldsymbol{n} e_{\tau, h} \ \text{d} l + \int_\tau  u\boldsymbol{v} \cdot \nabla e_{\tau, h} \ \text{d} \boldsymbol{x} - \int_\tau k \nabla \widetilde u_{\tau, h}  \cdot \nabla e_{\tau, h} \ \text{d} \boldsymbol{x},\\
\end{aligned}
\end{equation*}

\noindent
and $k_{\min, \tau}$ is the minimum value of $k$ over $\tau$. The boundedness of $k$ and the Cauchy-Schwarz inequality give
\begin{equation} \label{eq:pJ1}
J_1 \leq k_{\max, \tau} \| \nabla (w-u) \|_{L^2(\tau)}  \| \nabla e_{\tau, h} \|_{L^2(\tau)}
= k_{\max, \tau} R_{1,\tau} \| \nabla e_{\tau, h} \|_{L^2(\tau)},
\end{equation}

\noindent
where $k_{\max, \tau}$ is the maximum of $k$ over $\tau$. Also by Cauchy-Schwarz inequality and \eqref{eq:interr} we get
\begin{equation} \label{eq:pJ2}
J_2 \leq \| f \|_{L^2(\tau)} \| e_{\tau, h} - I_\tau e_{\tau, h} \|_{L^2(\tau)}  \leq Ch_\tau \| f \|_{L^2(\tau)} \| \nabla e_{\tau, h} \|_{L^2(\tau)}.
\end{equation}

\noindent
To estimate $J_3$, we first use \eqref{eq:bvpsum} and Green's formula to write
the first term in $J_3$ as
\begin{equation} \label{eq:fj3}
\begin{aligned}
\int_\tau f I_\tau e_{\tau, h} \ \text{d} \boldsymbol{x} &=
 \sum_{\xi\in v(\tau)} \int_{\partial t_\xi} ( -k \nabla \widetilde u_{\tau, h}  + u_h \boldsymbol{v})\cdot \boldsymbol{n} I_\tau e_{\tau, h}  \ \text{d} l \\
&= \sum_{\xi\in v(\tau)} \int_{t_\xi} \nabla \cdot ( -k \nabla \widetilde u_{\tau, h}  + u_h \boldsymbol{v}) I_\tau e_{\tau, h}  \ \text{d} \boldsymbol{x} \\
&=\int_{\tau} \nabla \cdot ( -k \nabla \widetilde u_{\tau, h}  + u_h \boldsymbol{v}) I_\tau e_{\tau, h}  \ \text{d} \boldsymbol{x}.
\end{aligned}
\end{equation}
Next, from \eqref{eq:pass1}, \eqref{eq:pass2},
and \eqref{eq:bvpbsum} the second term in $J_3$ is expressed as
\begin{equation} \label{eq:tj3}
\begin{aligned}
\int_{\partial\tau} ( k\nabla u - u\boldsymbol{v})  \cdot \boldsymbol{n} e_{\tau, h} \ \text{d} l &= a_\tau(u-u_h, e_{\tau, h}) - \int_{\partial \tau}  \widetilde g_\tau I_\tau e_{\tau, h} \ \text{d} l \\
&= a_\tau(u-u_h, e_{\tau, h}) - 
\int_{\partial \tau} ( - k \nabla \widetilde u_{\tau, h} \cdot \boldsymbol{n} + u_h \boldsymbol{v} \cdot \boldsymbol{n} ) I_\tau e_{\tau, h} \ \text{d} l.
\end{aligned}
\end{equation}
Furthermore, we apply integration by parts to the last term in $J_3$ to get
\begin{equation} \label{eq:sj3}
- \int_\tau k \nabla \widetilde u_{\tau, h}  \cdot \nabla e_{\tau, h} \ \text{d} \boldsymbol{x} = \int_\tau \nabla \cdot (k \nabla \widetilde u_{\tau, h} ) e_{\tau, h} \ \text{d} \boldsymbol{x} - 
\int_{\partial \tau} k \nabla \widetilde u_{\tau, h}  \cdot \boldsymbol{n} e_{\tau, h} \ \text{d} l.
\end{equation}
Putting \eqref{eq:fj3}, \eqref{eq:tj3}, and \eqref{eq:sj3} back into $J_3$, gives
\begin{equation*}
\begin{aligned}
J_3 = J_{31} + J_{32},
\end{aligned}
\end{equation*}
where
\begin{equation*}
\begin{aligned}
J_{31} & = \int_\tau  \nabla \cdot ( k \nabla \widetilde u_{\tau, h})  ( e_{\tau, h} - I_\tau e_{\tau, h} )  \ \text{d} \boldsymbol{x},  \\
J_{32} & =  a_\tau(u-u_h, e_{\tau, h}) 
+ \int_\tau  u\boldsymbol{v} \cdot \nabla e_{\tau, h} \ \text{d} \boldsymbol{x}
+  \int_\tau \nabla \cdot ( \boldsymbol{v} u_h ) I_\tau e_{\tau, h} \ \text{d} \boldsymbol{x} - \int_{\partial\tau} k\nabla \widetilde u_{\tau, h} \cdot \boldsymbol{n} e_{\tau, h} \ \text{d} l \\
& \quad  - \int_{\partial \tau} ( - k \nabla \widetilde u_{\tau, h} \cdot \boldsymbol{n} + u_h \boldsymbol{v} \cdot \boldsymbol{n} ) I_\tau e_{\tau, h} \ \text{d} l. \\
\end{aligned}
\end{equation*}

\noindent
By assuming sufficient smoothness of $k$ and using triangle inequality yields
\begin{equation} \label{eq:pJ31}
\begin{aligned}
J_{31} & \leq \| \nabla k \|_{L^\infty(\tau)}  \| \nabla \widetilde u_{\tau, h} \|_{L^2{(\tau)}}  \| e_{\tau,h} - I_\tau e_{\tau,h} \|_{L^2{(\tau)}}  \\
& \leq Ch_\tau \| \nabla k \|_{L^\infty(\tau)}  \| \nabla \widetilde u_{\tau, h} \|_{L^2{(\tau)}}  \| \nabla e_{\tau, h} \|_{L^2{(\tau)}}  \\
& \leq Ch_\tau \| \nabla k \|_{L^\infty(\tau)}  \| \nabla w \|_{L^2{(\tau)}}  \| \nabla e_{\tau, h} \|_{L^2{(\tau)}} \\
& \quad + Ch_\tau \| \nabla k \|_{L^\infty(\tau)}  \| \nabla e_{\tau, h} \|^2_{L^2{(\tau)}}.
\end{aligned}
\end{equation}

\noindent
Expanding the first and second term in $J_{32}$ and applying integration by parts
and using \eqref{pde} gives
\begin{equation} \label{eq:fsj32}
\begin{aligned}
a_\tau(u-u_h, e_{\tau, h}) 
+ \int_\tau  u\boldsymbol{v} \cdot \nabla e_{\tau, h} \ \text{d} \boldsymbol{x}
&= \int_\tau k \nabla(u-u_h) \cdot \nabla e_{\tau,h} \, \text{d} \boldsymbol{x}
- \int_\tau \nabla \cdot (u_h \boldsymbol{v}) e_{\tau,h} \, \text{d} \boldsymbol{x}\\
&\quad+ \int_{\partial \tau} u_h \boldsymbol{v} \cdot \boldsymbol{n} e_{\tau,h} \, \text{d} l\\
&\quad+ \int_\tau \delta ( \boldsymbol{v}  \cdot \nabla e_{\tau, h}) \big( f - \nabla \cdot ( -k\nabla u_h + \boldsymbol{v} u_h ) \big)   \ \text{d} \boldsymbol{x}.
\end{aligned}
\end{equation}
Putting \eqref{eq:fsj32} back to $J_{32}$ and appropriately rearranging the terms give the following decomposition
\begin{equation*}
J_{32} =  J_{41} + J_{42} + J_{43} + J_{44},
\end{equation*}

\noindent
where
\begin{equation*}
\begin{aligned}
J_{41} & = \int_\tau k \nabla (u - u_h) \cdot \nabla e_{\tau, h} \ \text{d} \boldsymbol{x}, \\
J_{42} & = \int_{\partial\tau} ( -k \nabla \widetilde u_{\tau, h} + u_h \boldsymbol{v} ) \cdot \boldsymbol{n} ( e_{\tau, h} - I_\tau e_{\tau, h} ) \ \text{d} l, \\
J_{43} & = \int_\tau \nabla \cdot ( \boldsymbol{v} u_h ) ( I_\tau e_{\tau, h} - e_{\tau, h} ) \ \text{d} \boldsymbol{x},\\
J_{44} & = \int_\tau \delta ( \boldsymbol{v}  \cdot \nabla e_{\tau, h}) \big( f - \nabla \cdot ( -k\nabla u_h + \boldsymbol{v} u_h ) \big)   \ \text{d} \boldsymbol{x}.
\end{aligned}
\end{equation*}

\noindent
Cauchy-Schwarz inequality gives
\begin{equation} \label{eq:pJ41}
J_{41} \leq k_{\max, \tau} \| \nabla (u - u_h) \|_{L^2{(\tau)}}  \| \nabla e_{\tau, h} \|_{L^2{(\tau)}} = 
k_{\max, \tau} R_{3,\tau} \nabla e_{\tau, h} \|_{L^2{(\tau)}}.
\end{equation}

\noindent
Using Lemma 6.1 in \cite{chatzipantelidis2002finite} yields 
\begin{equation} \label{eq:pJ42}
\begin{aligned}
J_{42} & \leq Ch_\tau \big( \| \nabla k \|_{L^\infty(\tau)}  \| \nabla \widetilde u_{\tau, h} \|_{L^2{(\tau)}} + \| \nabla \cdot (u_h \boldsymbol{v} ) \|_{L^2{(\tau)}} \big) \| \nabla e_{\tau, h} \|_{L^2{(\tau)}}  \\
& \leq Ch_\tau \| \nabla k \|_{L^\infty(\tau)}  \| \nabla w \|_{L^2{(\tau)}}  \| \nabla e_{\tau, h} \|_{L^2{(\tau)}} + Ch_\tau \| \nabla k \|_{L^\infty(\tau)}  \| \nabla e_{\tau, h} \|^2_{L^2{(\tau)}} \\
& \quad + Ch_\tau \| \nabla \cdot (u_h \boldsymbol{v} ) \|_{L^2{(\tau)}} \| \nabla e_{\tau, h} \|_{L^2{(\tau)}} .
\end{aligned}
\end{equation}

\noindent
Using  Cauchy-Schwarz inequality and \eqref{eq:interr}, we get
\begin{equation} \label{eq:pJ43}
J_{43}  \leq Ch_\tau \| \nabla \cdot (u_h \boldsymbol{v} ) \|_{L^2{(\tau)}} \| \nabla e_{\tau, h} \|_{L^2{(\tau)}}.
\end{equation}

\noindent
By appropriately choosing $\delta$ (see \cite{tezduyar2000finite, tezduyar2003stabilization} for justification), we bound $J_{44}$ by
\begin{equation} \label{eq:pJ44}
J_{44} \leq Ch_\tau \| \boldsymbol{v} \|_{L^2(\tau)} \| \nabla e_{\tau, h} \|_{L^2{(\tau)}}.
\end{equation}

\noindent
Based on the above derivation, \eqref{eq:p1} can be rewritten as
\begin{equation*}
k_{\min, \tau} \| \nabla e_{\tau, h} \|^2_{L^2(\tau)} \leq J_1 + J_2 + J_{31} + J_{41} + J_{42} + J_{43} + J_{44}.
\end{equation*}
We collect the like-terms of all the estimates \eqref{eq:pJ1} for $J_1$, \eqref{eq:pJ2} for $J_2$, \eqref{eq:pJ31} for $J_{31}$, \eqref{eq:pJ41} for $J_{41}$, \eqref{eq:pJ42} for $J_{43}$, \eqref{eq:pJ43} for $J_{43}$, and \eqref{eq:pJ44} to give
\begin{equation*}
k_{\min, \tau} \| \nabla e_{\tau, h} \|^2_{L^2(\tau)} \leq \big( C_1 R_{1,\tau} + C_2 R_{2,\tau} + C_3R_{3,\tau} \big) \| \nabla e_{\tau, h} \|_{L^2(\tau)} + C h_\tau \| \nabla e_{\tau, h} \|^2_{L^2(\tau)}.
\end{equation*}
Choosing sufficiently small $h_\tau$, we can combine the last term on the right hand
side of the last inequality to the left hand side to get
\begin{equation*}
\| \nabla e_{\tau, h} \|_{L^2(\tau)} \leq \big( C_1 R_{1,\tau} + C_2 R_{2,\tau} + C_3R_{3,\tau} \big),
\end{equation*}
which is then combined with \eqref{eq:p0} to give the desired result. 
\end{proof}

\newtheorem{thm}{Theorem}[section]
\begin{thm} \label{thm:perr}
Assume $u$ is the solution of \eqref{pde} and $\widetilde u_h$ is the post-proccessed solution \eqref{eq:ppsol} satisfying \eqref{eq:bvp}-\eqref{eq:bvpb}, then we have
$$
\| \nabla ( u - \widetilde u_h) \|_{L^2(\Omega)} \leq Ch,
$$
where $C$ is a constant independent of $h$.
\end{thm}

\begin{proof}
Let $w_h$ be the nodal interpolation of $u$ such that $w_h|_\tau = w \in V(\tau)$ as in Lemma \ref{lem:elemerr}. Then
\begin{equation*}
\begin{aligned}
\| \nabla ( u - \widetilde u_h) \|^2_{L^2(\Omega)} & = \sum_{\tau} \| \nabla ( u - \widetilde u_h) \|^2_{L^2(\tau)} \\
& \leq
C\Big( \sum_{\tau} R^2_{1, \tau} + \sum_{\tau} R^2_{2, \tau} + \sum_{\tau} R^2_{3, \tau} \Big) \\
&= C \Big( \Vert \nabla(u - w_h) \Vert^2_{L^2(\Omega)}
+ h^2 ( \Vert f \Vert^2_{L^2(\Omega)} + \Vert \boldsymbol{v} \Vert^2_{L^2(\Omega)}
+ \Vert \nabla \cdot (u_h \boldsymbol{v} \Vert^2_{L^2(\Omega)} )\\
&\quad+ \Vert \nabla(u - u_h) \Vert^2_{L^2(\Omega)} \Big).
\end{aligned}
\end{equation*}
Moreover, the SUPG/CGFEM solution $u_h$ satisfies $ \| u - u_h \|_{H^1(\Omega)} \leq Ch\| \|_{H^2(\Omega)}$ \cite{brenner2008mathematical, johnson1984finite, roos2008robust}. Also, interpolation $w_h$ has a standard property
that it converges to the first order with respect to $h$ to $u$
(see \cite{brenner2008mathematical}). All these combined give the desired result.
\end{proof}

\section{Numerical Experiments}\label{sec:num}

\noindent
To examine various aspects of the performance of the post-processing technique we consider mainly two different simulation scenarios. One of them is for studying the performance of the post-processing technique for continuous CGFEM for advection diffusion equation while the other is for studying the performance of the post-processing technique for SUPG for advection dominated diffusion equation. For these two different scenarios, we demonstrate that (1) we obtain conservative fluxes after applying the post-processing technique; (2) we confirm the convergence rates discussed in Section \ref{sec:ana} through examples in Section \ref{sec:cs}. Moreover, we numerically show the convergence behavior of the post-processed solution in terms of three different edge metrics. The second part of the experiment focuses on the performance of the post-processing technique applied to the numerical solutions for drift-diffusion equations in Section \ref{sec:drift}.

\subsection{Conservation Study} \label{sec:cs}

We consider three test problems in this section. In these three problems, we consider the domain $\Omega = (0, 1)^2$ and impose Dirichlet boundary conditions. In the implementation, the domain is discretized by using triangular elements. 

\textbf{Example 1.} Advection diffusion equation \eqref{pde} with $k=1, \boldsymbol{v} = (1, 1)^\top, g = 0, u = (x-x^2) (y-y^2),$ and $f$ is the function derived from \eqref{pde}.

\textbf{Example 2.} Advection dominated diffusion equation \eqref{pde} with $k = 0.01, \boldsymbol{v} = (1, 1)^\top, g = 0$,
$$
u = \Big( x - \frac{e^{ \frac{x} {k} } - 1}{e^{\frac{1} {k}} - 1 } \Big) \Big( y - \frac{e^{\frac{y} {k} - 1} }{e^{\frac{1} {k}} - 1 } \Big),
$$ 
and $f$ is the function derived from \eqref{pde}. The solution develops two boundary layers, one at the right and the other at the top side of the domain. 

\textbf{Example 3.} We consider the time-dependent advection dominated diffusion equation
\begin{equation}\label{tpde}
\begin{cases}
\begin{aligned}
\partial_t u + \nabla \cdot ( - k \nabla u + \boldsymbol{v} u ) & = f \quad \text{in} \quad \Omega\times(0, T)  \\
u & = 0 \quad \text{on} \quad \partial\Omega\times(0, T)  \\
u(\boldsymbol{x}, 0) & = u_0(\boldsymbol{x}), ~~\boldsymbol{x} \in \Omega,
\end{aligned}
\end{cases}
\end{equation}
with $k = 10^{-5}, \boldsymbol{v} = \big( y - 0.5 , 0.5 - x  \big)^\top, f = 0, g = 0$, and initial solution  		
\begin{equation*}
u_0(\boldsymbol{x}) = 
	\begin{cases}
	0, &  \text{if} \ (x - 0.25)^2 + (y - 0.5)^2 > r^2, \\
	1, &  \text{if} \ \ (x - 0.25)^2 + (y - 0.5)^2 \leq r^2.
	\end{cases}
\end{equation*}
The solution is the rotation and diffusion of an initial cylinder of height 1, radius $r = 0.2$ and center (0.25, 0.5). Since the diffusion coefficient $k$ is very small compared to the advection coefficient $\boldsymbol{v}$ ( the equation is advection dominated), we mainly see the rotation of the cylinder.  The time required for one complete revolution for the simulation is $T = 2 \pi$; see \cite{bochev2013new}.

To solve this problem numerically, we utilize SUPG formulation \eqref{eq:supg} for the spatial discretization  and apply backward Euler scheme for the temporal discretization of \eqref{tpde}. A description of SUPG for time-dependent advection diffusion reaction equations with small diffusion and proposals for the parameter $\delta$ are discussed in \cite{john2008finite}. We use a uniform partition for the time variable $0 = t^0 < t^1 < t^2 < \cdots < t^N = T.$ with $\Delta t = t^n - t^{n-1} = \frac{T}{N}$ and $u(t^0) = u(0) = u_0$ is the initial data. Given $u^{n-1}_h$, the backward Euler SUPG (BE-SUPG) is to  find $u^n_h \in V_h$ satisfying
\begin{equation} \label{eq:tsfem}
(u^n_h, w_h) + \Delta t ~a(u^n_h, w_h) = (u^{n-1}_h, w_h) + \Delta t ~\ell(w_h), \quad \forall \ w_h \in V_h.
\end{equation}
where $(\cdot, \cdot)$ is the usual $L^2(\Omega)$ inner product.
We discretize the domain $\Omega = (0, 1)^2$ into 128 squares in each direction
and each square is further divided into two triangles. The time step for the example is $\Delta t = 2\pi / 2000$ and the numerical solution completes one revolution in 2000 steps \cite{bochev2013new}. Figure \ref{fig:timed2} present snapshots of the BE-SUPG solutions of the example taken at $t=0, t=\pi/2, \pi, 3\pi/2,$ and $2\pi.$ They show how the initial cylinder of height 1, radius 0.2, and center (0.25, 0.5) is rotated under the given velocity field $\boldsymbol{v} = (y-0.5, 0.5-x)$ and the small diffusion $k = 1.0\times 10^{-5}.$ The artificial diffusion in SUPG contributes to the diffusion that we see in Figure \ref{fig:timed2}.

\begin{figure}
\centering
\includegraphics[width=4.5cm]{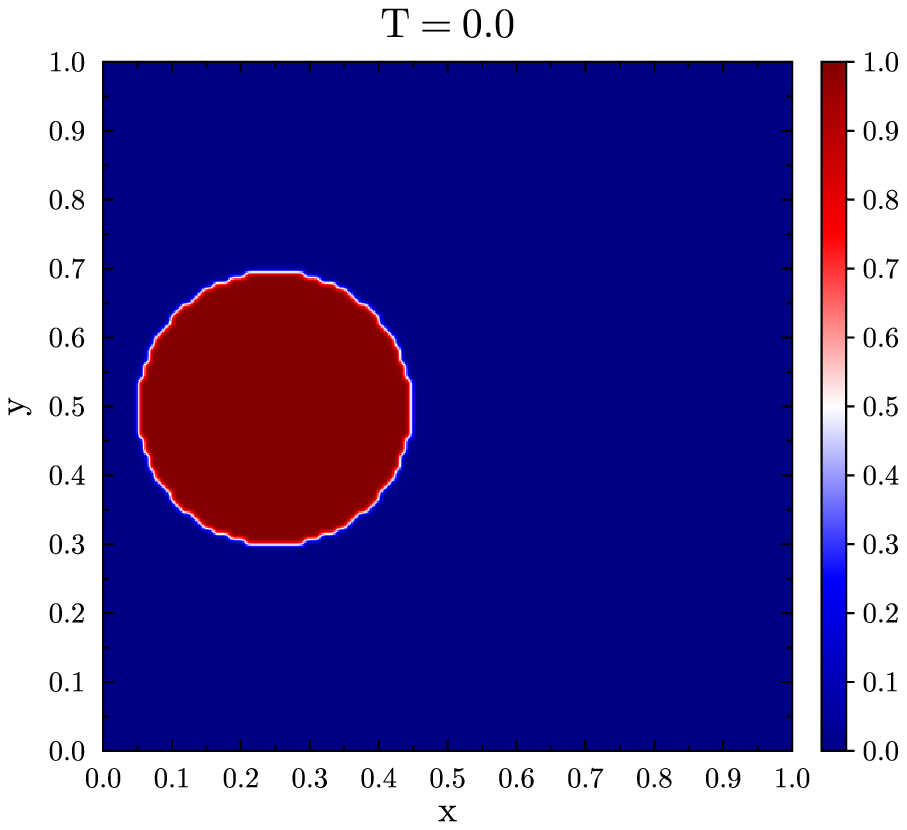} 
\includegraphics[width=4.5cm]{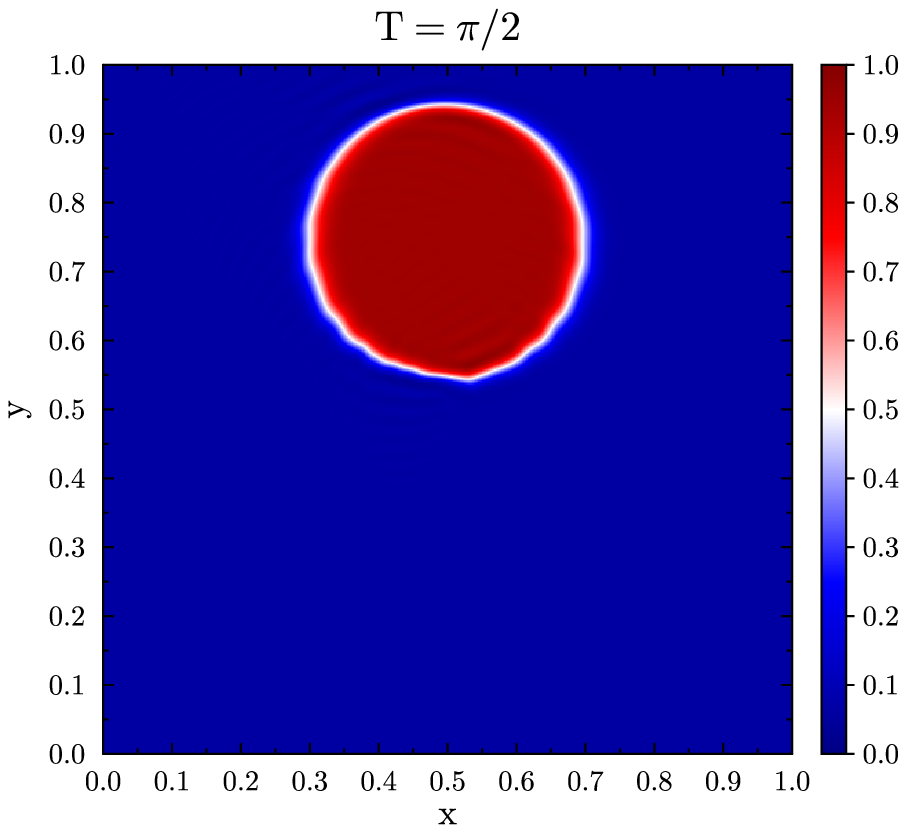} 
\includegraphics[width=4.5cm]{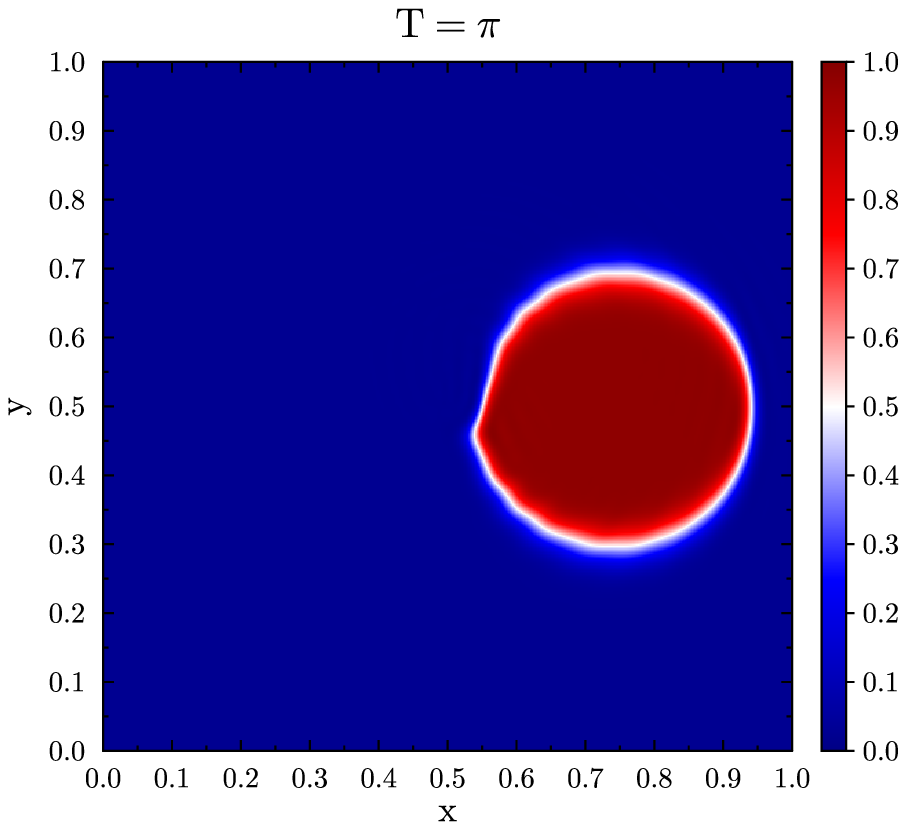} 
\includegraphics[width=4.5cm]{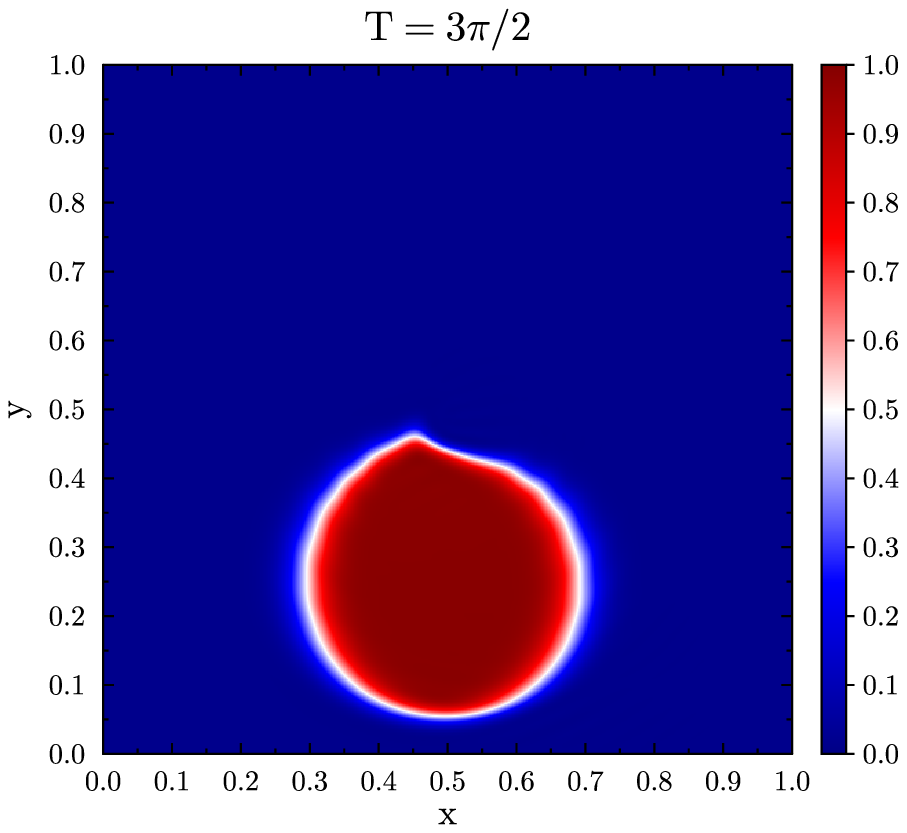} 
\includegraphics[width=4.5cm]{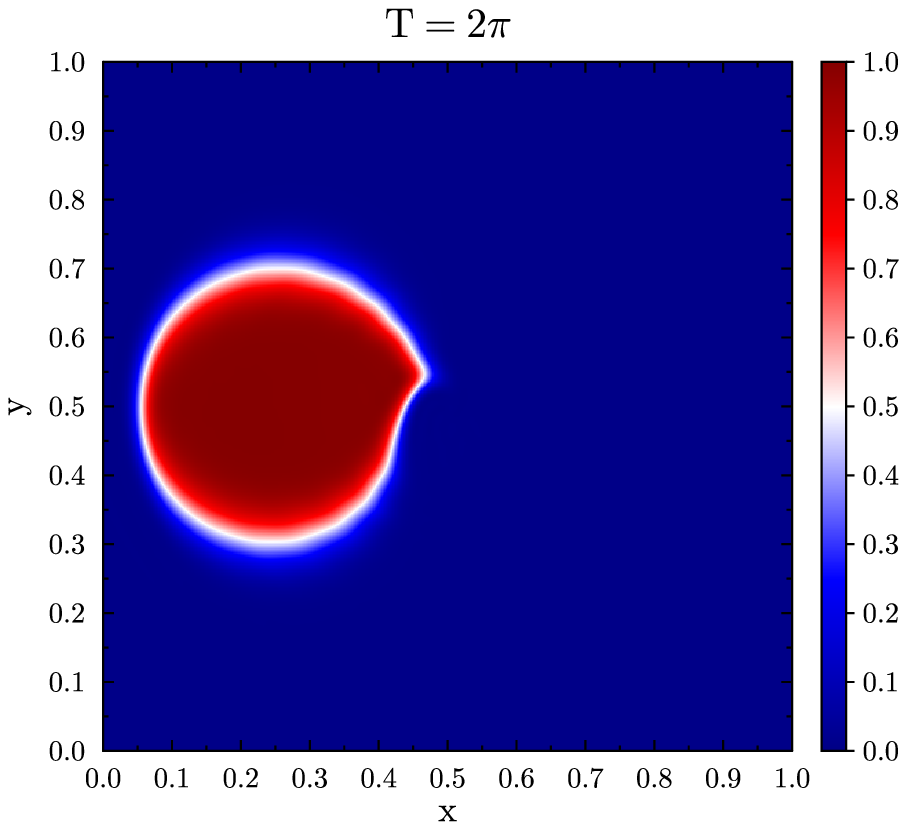} 
\caption{Contours of the BE-SUPG solutions of Example 3}
\label{fig:timed2}
\end{figure}

\noindent
The post-processing technique for BE-SUPG solutions to \eqref{tpde} is constructed in such a way that 
\begin{equation} \label{eq:tcvconservation}
\int_{\partial C^z} \widetilde {\boldsymbol{\nu}}_h \cdot \boldsymbol{n}  \ \text{d} l = \int_{C^z} ( f + \partial_t u_h ) \ \text{d} \boldsymbol{x}
\end{equation}
is satisfied.
The detailed procedure is a simple and natural extension of the procedures described in Section \ref{sec:pp}.

For the conservation study, we first consider the local conservation for the finite element solution and the post-processed solution. In this scenario, for Example 3, the local fluxes are collected at time $t=\pi.$
Figure \ref{fig:advncf} shows that without the post-processing technique, the local fluxes calculated directly from CGFEM, SUPG, and BE-SUPG solutions for Example 1, 2, and 3, respectively are not conservative since the errors are relatively large, where the error is defined as
\begin{equation}
\text{Error} = \int_{\partial C^z}  ( - k \nabla u_h \cdot \boldsymbol{n} + u_h \boldsymbol{v} \cdot \boldsymbol{n} ) \ \text{d} l - \int_{C^z} f \ \text{d} \boldsymbol{x},
\end{equation}
for Example 1 and Example 2 and
\begin{equation}
\text{Error} = \int_{\partial C^z}  ( - k \nabla u_h \cdot \boldsymbol{n} + u_h \boldsymbol{v} \cdot \boldsymbol{n} ) \ \text{d} l - \int_{C^z} (f + \partial_t u_h) \ \text{d} \boldsymbol{x},
\end{equation}
for Example 3.

\begin{figure}
\centering
\includegraphics[width=10.0cm]{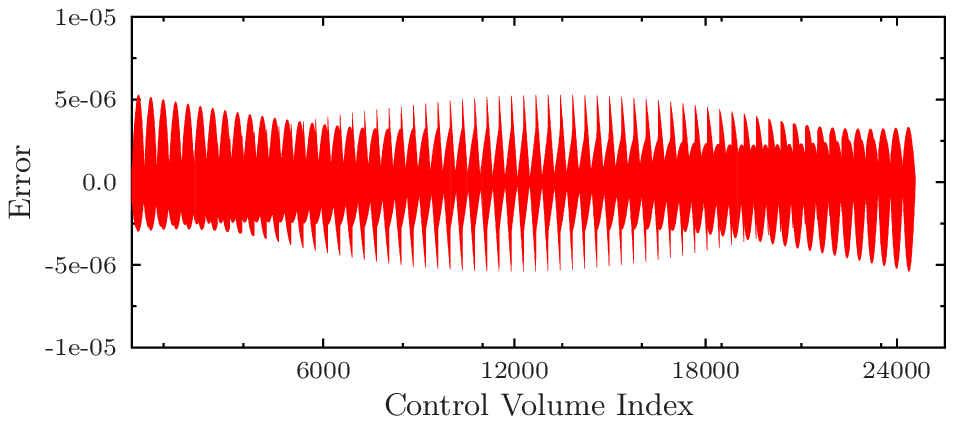}
\includegraphics[width=10.0cm]{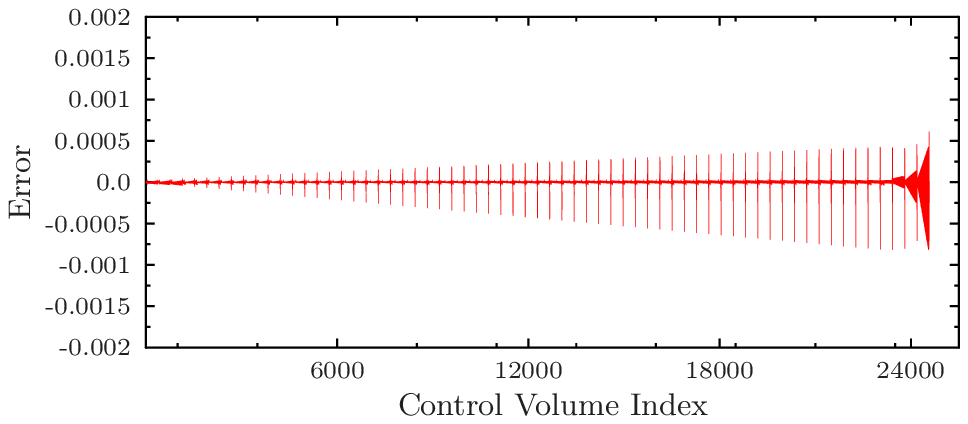}
\includegraphics[width=10.0cm]{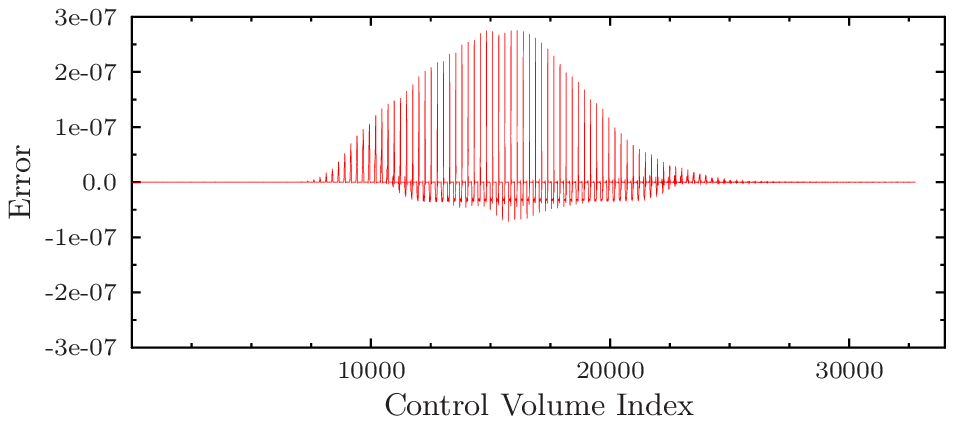}
\caption{Conservation errors without the post-processing technique for Example 1(top), Example 2(middle), and Example 3(bottom). }
\label{fig:advncf}
\end{figure}

After applying the post-processing technique, we have the local conservation. The conservation errors are of the scale of $10^{-18}$. Theoretically, the conservation errors after applying the post-processing technique should be zero according to the construction. These discrepancies are due to the accuracy of the linear algebra solver and the machine but are sufficient to verify the local conservation of the fluxes on each control volume.

Now for Example 1 and Example 2, we study the convergence rates. We first study the $H^1$ semi-norm of the CGFEM/SUPG solutions and the post-processed solutions. The results for Example 1 are shown in Figure \ref{fig:h1adverr1} while the results for Example 2 are shown in Figure \ref{fig:h1adverr2}. We solve Example 1 on $10\times 10, 20\times 20, 40\times 40, 80\times 80, 160\times  160,$ and $320\times 320$ uniform meshes while we solve Example 2  on $40\times 40, 80\times 80, 160\times  160,320\times 320, 640\times 640,$ and $1280\times 1280$ uniform meshes. In Figure \ref{fig:h1adverr1} and Figure \ref{fig:h1adverr2}, the plots on the left (red) show the $H^1$ semi-norm convergence order of the CGFEM for Example 1 and SUPG for Example 2. They are roughly 1's, which confirms the theoretical convergence rates of the numerical methods applied to both examples in literature \cite{brenner2008mathematical, johnson1984finite, roos2008robust}. The plots on right (blue) in Figure \ref{fig:h1adverr1} and Figure \ref{fig:h1adverr2} show the $H^1$ semi-norm errors of the post-processed solutions. Firstly, in both examples, the convergence rates in $H^1$ semi-norm of the post-processed solutions confirm the convergence analysis, i.e., Theorem \ref{thm:perr} in Section \ref{sec:ana}. Furthermore, we notice that the convergence rates in $H^1$ semi-norm of the post-processed solution is of the same order with the convergence rate in $H^1$ semi-norm of the finite element solution.

\begin{figure} 
\centering
\includegraphics[width=4.5cm]{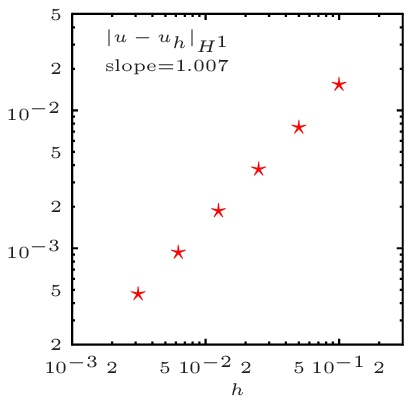}
\includegraphics[width=4.5cm]{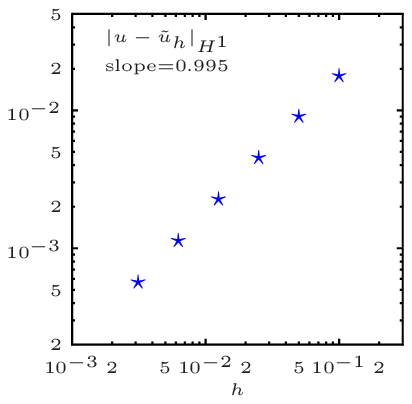}
\caption{$H^1$ semi-norm errors of the CGFEM solution and the post-processed solution for Example 1}
\label{fig:h1adverr1}
\includegraphics[width=4.5cm]{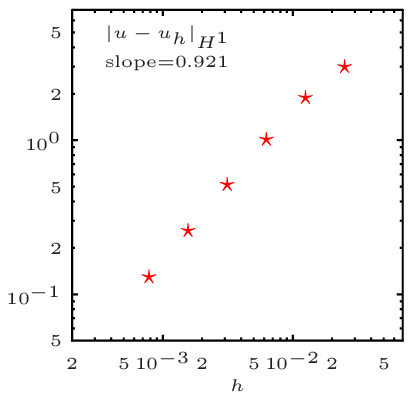}
\includegraphics[width=4.5cm]{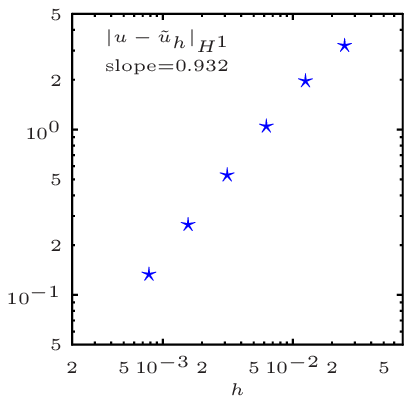}
\caption{$H^1$ semi-norm errors of the SUPG solution and the post-processed solution for Example 2}
\label{fig:h1adverr2}
\includegraphics[width=4.5cm]{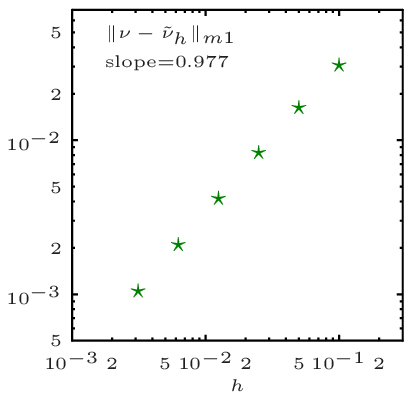}
\includegraphics[width=4.5cm]{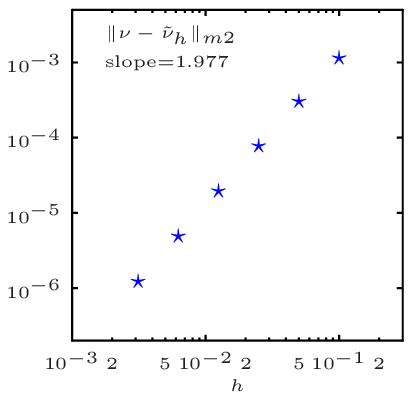}
\includegraphics[width=4.5cm]{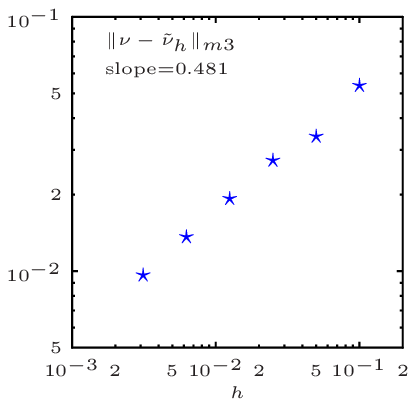}
\caption{Edge related metric errors of the post-processed solution for Example 1}
\label{fig:emadverr1}
\includegraphics[width=4.5cm]{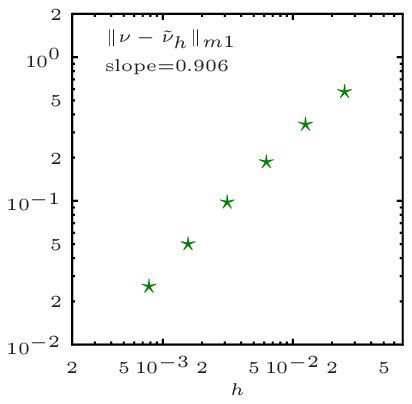}
\includegraphics[width=4.5cm]{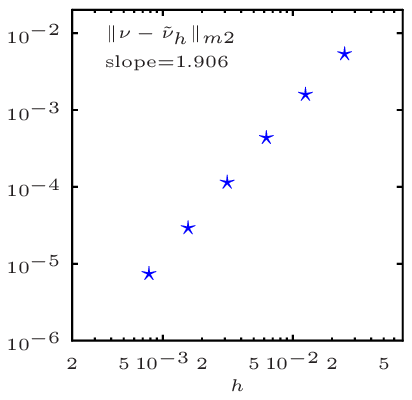}
\includegraphics[width=4.5cm]{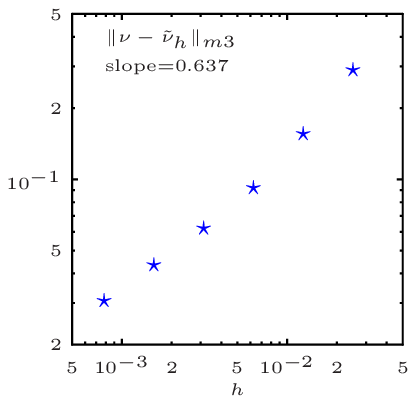}
\caption{Edge related metric errors of the post-processed solution for Example 2}
\label{fig:emadverr2}
\end{figure}

To further study the post-processed solutions, we collect the following errors which are calculated by taking the difference between the true fluxes and the post-processed fluxes at the boundary of each control volumes. Let $ \widetilde {\boldsymbol{\nu}}_{\tau, h} = -k \nabla \widetilde u_{\tau, h} + \boldsymbol{v} u_h $ be the post-processed velocity and $ \boldsymbol{\nu}_\tau = -k \nabla u + \boldsymbol{v} u $ be the true velocity on the element $\tau$ in the domain. Let $\mathcal{S}$ be the set of all the edges of all the control volumes and $e\in \mathcal{S}$ be an edge. We define the following edge related metric
$$
\| \boldsymbol{\nu} - \widetilde {\boldsymbol{\nu}}_h \|_{m1} = \max_{e\in \mathcal{S}} \| \widetilde {\boldsymbol{\nu}}_{\tau, h} \cdot \boldsymbol{n} - \boldsymbol{\nu}_\tau \cdot \boldsymbol{n}  \|_{L^\infty(e)}, 
$$
$$
\| \boldsymbol{\nu} - \widetilde {\boldsymbol{\nu}}_h \|_{m_2} = \max_{e\in \mathcal{S}} \Big{\|} \int_e ( \widetilde {\boldsymbol{\nu}}_{\tau, h} \cdot \boldsymbol{n} - \boldsymbol{\nu}_\tau \cdot \boldsymbol{n} ) \ \text{d} l \Big{\|}_{L^\infty(e)},
$$
and
$$
\| \boldsymbol{\nu} - \widetilde {\boldsymbol{\nu}}_h \|_{m3} = \Big( \displaystyle\sum_{e\in\mathcal{S}} \int_e \big( \widetilde {\boldsymbol{\nu}}_{\tau, h} \cdot \boldsymbol{n} - \boldsymbol{\nu}_\tau \cdot \boldsymbol{n}  \big)^2 \ \text{d} l \Big)^{1/2}.
$$

We collect the edge errors in these three different metrics and plot them to numerically show the convergence rates. These three metrics represent the errors of the post-processed normal flux at the boundaries of the control volumes in three different ways. The results are shown in Figure \ref{fig:emadverr1} and Figure \ref{fig:emadverr2}. They show that for both Example 1 and Example 2, the post-processed normal fluxes approaches its true counterpart in reasonable convergence rates as the mesh gets finer.

\subsection{Application to Drift-Diffusion Equations} \label{sec:drift}
Drift-Diffusion Equations is a coupled system of three equations of which one is an elliptic equation and the other two are advection dominated diffusion equations. It is a coupled system modelling the motion of electrons and holes in semiconductor materials \cite{mock1983analysis, markowich1986stationary}. The governing equations are
\begin{equation} \label{drift}
\begin{cases}
\begin{aligned}
- \nabla \cdot (\lambda^2 \nabla\psi) & = p - n + C \quad \text{in} \ \Omega \\
 \nabla \cdot ( - n \mu_n \nabla\psi + D_n \nabla n ) &= R(\psi, n, p)  \quad \text{in} \ \Omega \\
  \nabla \cdot ( p \mu_p \nabla\psi + D_p \nabla p ) & = R(\psi, n, p)   \quad \text{in} \ \Omega, \\
\end{aligned}
\end{cases}
\end{equation}

\noindent
with appropriate boundary conditions, where $\psi$ is the electric potential, $n$ and $p$ are carrier densities of the electrons and the holes, $\lambda$ is the minimal Debye length of the device, $D_n$ and $D_p$ are carrier's diffusivity, $\mu_n$ and $\mu_p$ are the carrier's mobilities, $R$ is the recombination term, and $C$ is a constant. The existence and uniqueness of the coupled system is discussed in \cite{mock1983analysis, gajewski1994uniqueness}. Numerical schemes for solving this coupled system need (1) stability when drift velocities dominate their diffusivity and (2) local conservation of electron and hole current densities \cite{bochev2013new}. We use CGFEM to solve the first equation in \eqref{drift} and SUPG to solve the other two equations. To obtain the local conservation for electron and hole current densities, we apply the post-processing technique to the numerical solutions.

A simple test example is presented to illustrate and verify the post-processing technique. We consider  equations \eqref{drift} with
\begin{itemize}
\item $\Omega = [0, 1]\times[0, 1]$
\item $C=0, \lambda =1$, $ \mu_n = 1,  \mu_p = 1$, $D_n = 0.01, D_p = 0.01$
\item $\psi = x + y, n = p = \Big( x - \frac{e^{ \frac{x} {k} } - 1}{e^{\frac{1} {k}} - 1 } \Big) \Big( y - \frac{e^{\frac{y} {k} - 1} }{e^{\frac{1} {k}} - 1 } \Big)$, where $k = 0.01.$
\end{itemize}

The boundary conditions and the function $R$ are derived from the true solutions. The problem is solved on $80\times 80, 160\times 160, 320\times 320,$ and $640\times 640$ uniform meshes. Figure \ref{fig:driftn} shows that the local current and hole densities are not conservative without applying the post-processing. After applying the post-processing, these errors are of scale $10^{-18}$ which is sufficient to claim that the local current and hole densities are conservative. The numerical results of the $H^1$ semi-norm errors of both the SUPG solutions and the post-processed solutions as well as their $\| \cdot \|_{m1}$ metric are shown in Figure \ref{fig:drift}. The convergence rates in these plots confirm the numerical discussion in Section \ref{sec:cs}, which further confirm the analysis in Section \ref{sec:ana}.

\begin{figure} 
\centering
\includegraphics[width=10.0cm]{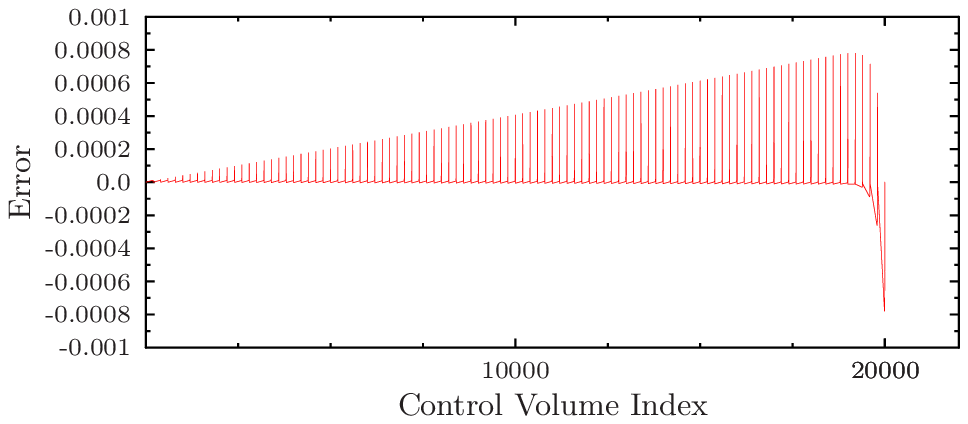}
\includegraphics[width=10.0cm]{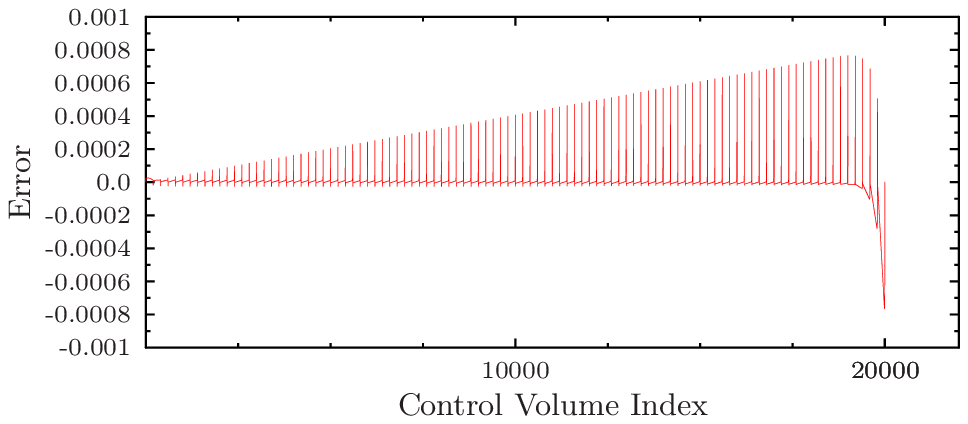}
\caption{Conservation errors without the post-processing technique for $n$(top) and $p$(bottom).}
\label{fig:driftn}
\end{figure}

\begin{figure} 
\centering
\includegraphics[width=4.5cm]{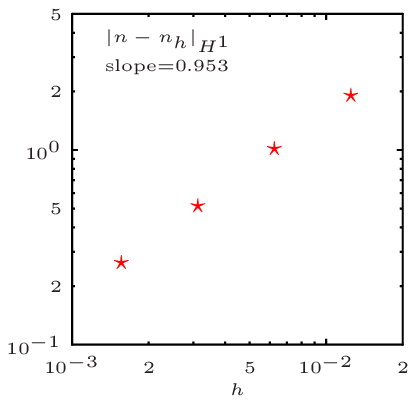}
\includegraphics[width=4.5cm]{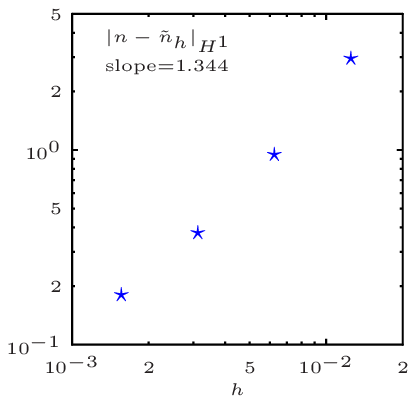}
\includegraphics[width=4.5cm]{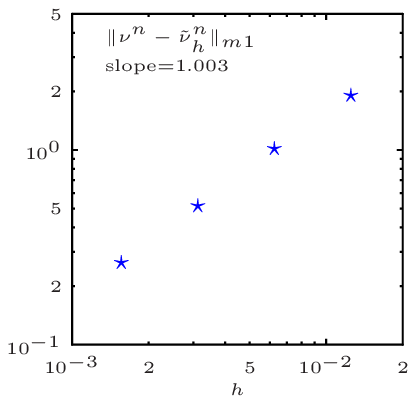}
\includegraphics[width=4.5cm]{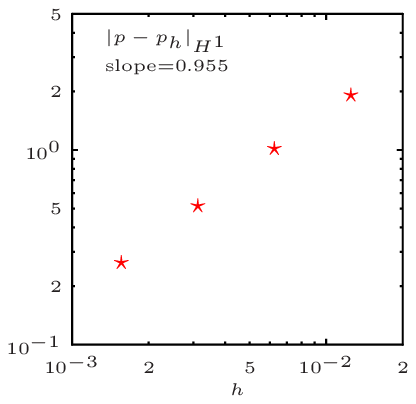}
\includegraphics[width=4.5cm]{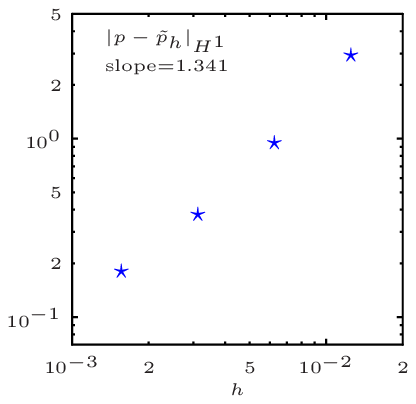}
\includegraphics[width=4.5cm]{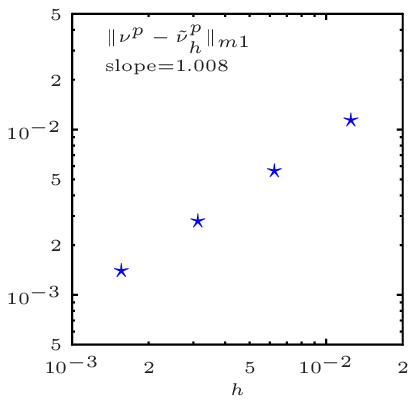}
\caption{Numerical errors for $n$ (top) and $p$ (bottom)}
\label{fig:drift}
\end{figure}

\section{Conclusion}
We proposed a post-processing technique for the CGFEM and SUPG for solving advection diffusion and advection dominated problems. For application problems where the conservation property is crucial, standard finite element solutions are
not adequate because naive calculation of the fluxes from those solutions are not locally conservative. Due to the popularity of standard finite elements, a post-processing for the finite element method is beneficial in this situation. By applying the post-processing technique proposed in this paper, conservative fluxes are obtained on the control volumes which are built on the dual mesh of the finite element mesh. Both analysis and numerical simulations are presented to illustrate the performance of the post-processing technique. The post-processing technique is quite simple and it is easy to implement it  in parallel computing environment.

\bibliographystyle{siam}
\bibliography{ppsupg}

\end{document}